\documentclass[a4paper]{article}
\pdfoutput=1

\usepackage{fullpage}
\usepackage[T1]{fontenc}
\usepackage{graphicx}
\usepackage{amsmath,amssymb,amsthm}
\usepackage{mathtools}
\usepackage{microtype}
\usepackage{xcolor}
\usepackage[colorlinks = true, linkcolor = black, citecolor = black]{hyperref}
\usepackage{algorithm}
\usepackage{algorithmicx}
\usepackage{algpseudocode}
\usepackage{subcaption}
\usepackage[group-separator={,}]{siunitx}
\algnewcommand\algorithmicinput{\textbf{Input:}}
\algnewcommand\Input{\item[\algorithmicinput]}
\algnewcommand\algorithmicoutput{\textbf{Output:}}
\algnewcommand\Output{\item[\algorithmicoutput]}

\newcommand{\R}{\mathbb{R}}
\newcommand{\bA}{\mathbf{A}}
\newcommand{\bS}{\mathbf{S}}
\newcommand{\bP}{\mathbf{P}}
\newcommand{\caC}{\mathcal{C}}
\DeclareMathOperator{\tr}{tr}
\DeclareMathOperator{\diag}{diag}

\DeclareMathOperator{\GL}{GL}
\DeclareMathOperator{\argmin}{argmin}
\DeclareMathOperator{\err}{err}
\DeclarePairedDelimiter{\abs}{\lvert}{\rvert}
\DeclarePairedDelimiter{\norm}{\lVert}{\rVert}

\newtheorem{theorem}{Theorem}[section]
\newtheorem{lemma}[theorem]{Lemma}
\newtheorem{proposition}[theorem]{Proposition}
\theoremstyle{definition}
\newtheorem{definition}[theorem]{Definition}
\newtheorem{assumption}[theorem]{Assumption}

\numberwithin{equation}{section}

\title{Accelerating operator Sinkhorn iteration with overrelaxation}
\author{Tasuku Soma\thanks{The Institute of Statistical Mathematics, Tokyo, 190-8562, Japan} \and Andr\'e Uschmajew\thanks{Institute of Mathematics \& Centre for Advanced Analytics and Predictive Sciences, University of Augsburg, 86159 Augsburg, Germany}}
\date{}

\begin{document}

\maketitle
\begin{abstract}
We propose accelerated versions of the operator Sinkhorn iteration for operator scaling using successive overrelaxation. We analyze the local convergence rates of these accelerated methods via linearization, which allows to determine the asymptotically optimal relaxation parameter based on Young's SOR theorem. Using the Hilbert metric on positive definite cones, we also obtain a global convergence result for a geodesic version of overrelaxation in a specific range of relaxation parameters. These techniques generalize corresponding results obtained for matrix scaling by Thibault et al.~(\textit{Algorithms}, 14(5):143, 2021) and Lehmann et al.~(\textit{Optim.~Lett.}, 16(8):2209–2220, 2022). Numerical experiments demonstrate that the proposed methods outperform the original operator Sinkhorn iteration in certain applications. 
\end{abstract}

\section{Introduction}

In \emph{operator scaling}, we are given matrices $A_1, \dots, A_k \in \R^{m \times n}$ and the goal is to find a joint transformation
\[
 \bar A_i = L A_i R^\top
\]
by square invertible matrices $L \in \GL_m(\R)$ and $R \in \GL_n(\R)$ such that
\begin{align}
 \sum_{i=1}^k \bar A_i^{} \bar A_i^\top  &= \sum_{i=1}^k  L A_i^{} R^\top R  A_i^\top L^\top =  \frac{1}{m}I_m \label{eq: OSa} \\
 \intertext{and}
 \sum_{i=1}^k \bar A_i^\top \bar A_i^{} &= \sum_{i=1}^k R A_i^\top L^\top L A_i R^\top = \frac{1}{n}I_n. \label{eq: OSb}
\end{align}
The operator scaling problem has been first introduced by Gurvits~\cite{Gurvits2004} (in the complex setting) and has later found various applications in non-commutative polynomial identity testing~\cite{Garg2019}, computational invariant theory~\cite{Allen-Zhu2018}, functional analysis~\cite{Garg2018,Sra2018}, scatter estimation~\cite{Franks2020,Drton2021}, and signal processing~\cite{Barthe1998,Kwok2018}. A special case of operator scaling is well-known \emph{matrix scaling}~\cite{Sinkhorn1967} and various properties of matrix scaling can be generalized to operator scaling; see the survey paper by Idel~\cite{Idel2016} for more details.

Note that \emph{scaling}, that is, the left and right multiplication by $L$ and $R$, can be understood as a group action of $\GL_m(\R) \times \GL_n(\R)$ on the space $(\R^{m \times n})^k$ of matrix tuples $(A_1,\dots,A_k)$. It will be useful to introduce a notation $\bA = (A_1,\dots,A_k)$ for the elements of $(\R^{m \times n})^k$ and denote by $
L \bA R^\top = (L A_1 R^\top, \dots, L A_k R^\top)$ the action of $(L,R) \in \GL_m(\R) \times \GL_n(\R)$ on $\bA$. Operator scaling then asks for a member $\bar{\bA}$ in the orbit of~$\bA$ that satisfies~\eqref{eq: OSa}--\eqref{eq: OSb}. Thus, we can regard~\eqref{eq: OSa}--\eqref{eq: OSb} either as a nonlinear system of equations in $(\R^{m \times n})^k$ for finding $\bar{\bA}$ in the orbit of $\bA$, or as a system of equations in $\GL_m(\R) \times \GL_n(\R)$ for finding scaling matrices $L$ and~$R$.

The \emph{operator Sinkhorn iteration}~\cite{Gurvits2004} is a simple iterative method for solving the operator scaling problem. Generalizing the classical Sinkhorn iteration~\cite{Sinkhorn1967} for matrix scaling, it is based on the observation that each of the two conditions~\eqref{eq: OSa} and~\eqref{eq: OSb} alone is easy to satisfy through one-sided scaling, provided that $\sum_{i=1}^k A_i A_i^\top$ and $\sum_{i=1}^k A_i^\top A_i$ are invertible. The method hence proceeds by solving the equations in an alternating manner until a fixed point is reached. Depending on the viewpoint onto the problem, the operator Sinkhorn iteration can be presented in various different ways: as a fixed-point iteration either for sequences $(L_t,R_t)$ of scaling matrices, a fixed-point iteration of positive definite matrices $(X_t, Y_t) = (L_t^\top L_t^{}, R_t^\top R_t^{})$, or as a fixed-point iteration for sequences $(\bA_t)$ in the orbit of $\bA$. These different versions of the algorithm, however, turn out to be essentially equivalent (see Section~\ref{sec: operator Sinkhorn}).

As will be explained in Section~\ref{sec: alternating optimization}, the operator scaling problem can also be studied from the perspective of optimization. Indeed, the solution $(X,Y) = (L^\top L, R^\top R)$ can be characterized as the minimizer of a geodesically convex function on cones of positive definite matrices. As such, the operator scaling problems actually constitutes a cornerstone in the area of \emph{noncommutative optimization}. In this viewpoint, the operator Sinkhorn iteration then can be interpreted as an interesting instance of an alternating minimization method.

\paragraph*{Our contribution.}
The goal of this work is to study \emph{accelerated} versions of the operator Sinkhorn iteration based on \emph{successive overrelaxation} (SOR). This is a classic approach both for linear and nonlinear systems of equations. Notably, for systems of linear equations, the SOR method is a linear fixed-point iteration with a relaxation parameter that combines old and new iterates and can be optimized (under certain conditions) for achieving the optimal convergence rate; see,~e.g.,~\cite[Chapters~3 and~4]{Hackbusch2016}. For nonlinear fixed-point iterations, similar ideas can be applied, but the accelerated convergence can typically be proved only locally based on linearization; see,~e.g.,~\cite{Schechter1962,Ortega1970,Hageman1975} for some classic references.

In the main section of this work (Section~\ref{sec: overrelaxation}), we propose several variants of overrelaxation for the operator Sinkhorn iteration. Each variant corresponds to a particular choice of the underlying space on which we ``combine'' old and new iterates. The most straight-forward way is taking affine combinations of iterates, that is, $(L_t, R_t)$ or $(X_t, Y_t)$, in the Euclidean sense. A slightly more sophisticated way is first applying a coordinate transform (e.g., matrix logarithm) to iterates and then taking affine combinations in the new coordinate. Finally, one may take a more geometric approach using geodesics on the space of positive definite matrices with respect to the well-known \emph{Hilbert (projective) metric}~\cite{Lemmens2012} for forming ``affine combinations'' along geodesics. Each variant has different characteristics, e.g., its domain and iteration complexity. Nevertheless, we show that near a fixed-point these variants are actually equivalent in the first-order sense, which allows to analyze their \emph{local} convergence in a unified way. We will deduce a familiar formula for the spectral radius of the linearized SOR iteration which then allows to determine the (asymptotically) optimal relaxation parameter along known lines (Theorem~\ref{thm: local convergence}). While our approach to analyzing the local convergence of nonlinear SOR is by no means new, and in fact follows a similar pattern as in recent work~\cite{Thibault2021,Lehmann2022} on matrix scaling, executing it rigorously for the operator Sinkhorn iteration requires several nontrivial technical steps, which we believe are of independent interest and provide additional insight into the problem.

The question of \emph{global} convergence is challenging in general for nonlinear SOR methods. For the geodesic version of our SOR methods we are able to show the global convergence for a limited range of the relaxation parameter (Theorem~\ref{thm: global convergence overrelaxation}). The starting point for this result is the contractivity of the standard Sinkhorn iteration in the Hilbert metric (under certain assumptions on $A_1,\dots,A_k$), which is classic for matrix scaling~\cite{Franklin1989} and extended recently for operator scaling~\cite{Georgiou2015,Idel2016}. In~\cite{Lehmann2022} it has been shown for matrix scaling that the contractivity is preserved when using underrelaxation or very mild overrelaxation. Extending this results to the SOR operator Sinkhorn iteration is not completely straightforward. For example, the analysis in~\cite{Lehmann2022} relies on the fact that in matrix scaling the underlying Hilbert metric can be expressed as the total variation semi-norm in the auxiliary log-space. To our knowledge, a similar property does not hold in operator scaling. To this end, we develop a novel metric inequality of the Hilbert metric for positive definite matrices, which allows us to bypass the auxiliary log-space and apply the technique of~\cite{Lehmann2022} nevertheless. To the best of our knowledge, this metric inequality (Lemma~\ref{lem:Hilbert-omega}) is not known in the literature and might be of independent interest.

Finally, we conduct numerical experiments to demonstrate the effectiveness of our accelerated methods, based on an adaptive choice of an almost optimal relaxation parameter. For an exemplary application in frame scaling, we show that the SOR methods achieve several magnitudes faster convergence than the original operator Sinkhorn iteration. To provide a complete picture, we also discuss the numerical behaviour of the SOR algorithms for operator scaling with ill-conditioned input matrices.

\paragraph*{Previous work.}
The operator Sinkhorn iteration was proposed by Gurvits~\cite{Gurvits2004} and he showed its global convergence for a class of input matrices called the Edmonds--Rado class. The global convergence of the general case was resolved by~\cite{Garg2019}; they showed that the operator Sinkhorn iteration converges in $O(1/\sqrt{t})$ rate, where $t$ is the number of iterations. Allen-Zhu~et~al.~\cite{Allen-Zhu2018} proposed the box-constrained Newton method for operator scaling and showed its global linear convergence. However, the per-iteration computational complexity is much higher than for operator Sinkhorn iteration. The Hilbert metric and the closely related Thompson metric have been applied to the analysis of fixed-point iterations in matrix variables~\cite{Lemmens2012,Sra2013,Sra2015,Georgiou2015}. Recently, Weber and Sra~\cite{Weber2023} applied the Thompson metric to fixed-point iterations for computing the Brascamp-Lieb constant, which can be reduced to operator scaling~\cite{Garg2018}. Our work generalizes~\cite{Gabriel2019,Thibault2021,Lehmann2022} where relaxation has been proposed for accelerating matrix scaling, together with a local, and in case of~\cite{Lehmann2022} also global, convergence analysis. The technical difficulty in the local analysis is taking a certain group equivariance of the problem into account which leads to non-positive definite Hessians. This methodology has been formalized in~\cite{Oseledets2023}, which contains an explicit version of Young's SOR theorem for $2\times 2$ positive semidefinite block systems. Our proposed heuristic for choosing a near optimal relaxation parameter already appeared in~\cite{Hageman1975} and was used in~\cite{Lehmann2022,Oseledets2023} as well.

\paragraph*{Outline.}
The rest of this paper is organized as follows. Section~\ref{sec: operator Sinkhorn} reviews the standard operator Sinkhorn iteration and its equivalent descriptions as a fixed-point iteration for the scaling matrices $(L,R)$, as a fixed-point iteration on the space of positive definite matrices, or as an alternating minimization algorithm. In addition, the basics of the Hilbert metric are introduced. In the main section (section~\ref{sec: overrelaxation}) several versions for overrelaxation of the operator Sinkhorn iteration are introduced (Sections~\ref{sec: overrelaxation Euclidean}--\ref{sec: overrelaxation Hilbert metric}) and their local convergence is studied (Section~\ref{sec: local convergence}). For a geodesic SOR version the global convergence in the Hilbert metric for a limited range of relaxation parameters is established as well (Section~\ref{sec: global convergence}). Section~\ref{sec: numerical experiments} presents numerical experiments and Section~\ref{sec: conclusion} contains a conclusion.

\section{Operator Sinkhorn iteration}\label{sec: operator Sinkhorn}

The operator Sinkhorn algorithm is based on the fact that it is very easy to satisfy only one of the relations in~\eqref{eq: OSa} and~\eqref{eq: OSb} by an appropriate left or right scaling. The working assumption for this is:

\begin{assumption}\label{assumption 1}
The matrices $\sum_{i=1}^k A_i^{} A_i^\top$ and $\sum_{i=1}^k A_i^\top A_i^{}$ are invertible.
\end{assumption}
This is indeed a necessary condition for the existence of a solution to the operator scaling problem.
To see this, note that left and right scalings preserve the invertibility of $\sum_{i=1}^k A_i^{} A_i^\top$ and $\sum_{i=1}^k A_i^\top A_i^{}$.
So, if one of these matrices is singular, then so is the scaled one and hence it cannot be equal to the (normalized) identity matrix. Under this assumption, consider a factorization
\[
 \sum_{i=1}^k A_i^{} A_i^\top = C C^\top
\]
with invertible $C \in \GL_m(\R)$, then choosing
\begin{equation}\label{eq: update rule for L}
 L = \frac{1}{\sqrt{m}} C^{-1}
\end{equation}
leads to the desired equation
\[
 \sum_{i=1}^k (L A_i^{}) (L A_i)^\top = \frac{1}{m} C^{-1} \left( \sum_{i=1}^k A_i^{} A_i^\top \right) C^{-\top} = \frac{1}{m} C^{-1} C C^\top C^{-\top} = \frac{1}{m} I_m.
\]
Similarly, decomposing
\[
 \sum_{i=1}^k A_i^{\top} A_i^{} = D D^\top
\]
with invertible $D \in \GL_n(\R)$ and setting
\begin{equation}\label{eq: update rule for R}
 R = \frac{1}{\sqrt{n}} D^{-1}
\end{equation}
immediately yields
\[
 \sum_{i=1}^k (A_i R^\top)^{\top} (A_i^{} R^\top) = \frac{1}{n} D^{-1} \left( \sum_{i=1}^k A_i^{\top} A_i^{} \right) D^{-\top} = \frac{1}{n} D^{-1} D D^\top D^{-\top} = \frac{1}{n} I_n.
\]

As the goal is to satisfy both relations simultaneously, the operator Sinkhorn algorithm produces sequences $(L_t)$, $(R_t)$ in an alternating way. Here, several versions are possible. A natural first version is denoted as Algorithm~\ref{algo: OSI}. It is not difficult to show that under Assumption~\ref{assumption 1} all steps in this algorithm are feasible (in exact arithmetic). In other words, Assumption~\ref{assumption 1} remains valid for $\bar{\bA}_t$ for all $t$. The proof (by induction) can be based on the following observation.

\begin{algorithm}[t]
\small
\caption{Operator Sinkhorn Iteration (OSI)}
\label{algo: OSI}
\begin{algorithmic}[1]
\Input
Initial matrix collection $\bar{\bA}_0 = (\bar A_{0,1},\dots,\bar A_{0,k}) = \bA$ satisfying Assumption~\ref{assumption 1}
\For
{$t = 0,1,2,\dots$}
    \State
    Compute a decomposition
    \[
    \sum_{i=1}^k \bar A_{t,i}^{} \bar A_{t,i}^\top = \bar C_t^{} \bar C_t^\top
    \]
    with $\bar C_t \in \GL_m(\R)$ and set
    \[
    \bar L_{t+1} = \frac{1}{\sqrt{m}} \bar C_t^{-1}.
    \]
    \State
    Compute a decomposition
    \[
    \sum_{i=1}^k \bar A_{t,i}^{\top} \bar L_{t+1}^\top \bar L_{t+1}^{} \bar A_{t,i}^{} = \bar D_t \bar D_t^{\top}
    \]
    with $\bar D_t \in \GL_n(\R)$ and set
    \[
    \bar R_{t+1} = \frac{1}{\sqrt{n}} \bar D_t^{-1}.
    \]
    \State 
    Set
    \[
    \bar{\bA}_{t+1} = \bar L_{t+1}^{} \bar{\bA}_t^{} \bar R_{t+1}^\top.
    \]
\EndFor
\end{algorithmic}
\end{algorithm}

\begin{lemma}\label{lem: assumption 1}
Given matrices $B_1,\dots,B_k$ such that $\sum_{i=1}^k B_i^\top B_i$ is positive definite. Then for any symmetric positive definite $M$ (of appropriate size) the matrix $\sum_{i=1}^k B_i^\top M B_i$ is positive definite  as well.
\end{lemma}

\begin{proof}
Assuming this is false, there exists $x \neq 0$ such that $\sum_{i=1}^k x^\top B_i^\top M B_i x = 0$. Since $M$ is positive definite,  this implies $B_i x = 0$ for every $i = 1,\dots,k$. Then $x^\top \left( \sum_{i=1}^k B_i^\top B_i \right) x = \sum_{i=1}^k x^\top B_i^\top B_i x = 0$, contradicting the assumption.
\end{proof}

Note that the computed $\bar{\bA}_t$ is related to the original $\bA$ by
\[
 \bar{\bA}_{t}^{} = L_t^{} \bA^{} R_t^\top
\]
with
\begin{equation}\label{eq: formula Lt Rt}
 L_t = \bar L_t \bar L_{t-1} \cdots \bar L_1, \qquad R_t = \bar R_t \bar R_{t-1} \cdots \bar R_1.
\end{equation}
If the final scaling matrices $L_t$ and $R_t$ are required for further use, they hence can be computed on the fly when executing Algorithm~\ref{algo: OSI} using recursive updates $L_{t+1} = \bar L_{t+1} L_t$ and $R_{t+1} = \bar R_{t+1} R_t$.

\subsection{Fixed-point iteration for scaling factors}
Algorithm~\ref{algo: OSI} appears difficult to analyze as an iteration for the scaling matrices $\bar L_t$ and $\bar R_t$ because the matrices $\bar A_{t,1},\dots,\bar A_{t,k}$ in the update formulas~\eqref{eq: update rule for L} and~\eqref{eq: update rule for R} change in every step. It should hence be rather regarded and analyzed as an iteration for the scaled matrices $\bar{\bA}_t$, a path that however we do not follow here. Instead, we consider a modified version denoted as Algorithm~\ref{algo: FFPI}, in which the matrices $A_{1},\dots,A_k$ remain unaltered.

\begin{algorithm}[t]
\small
\caption{Factor Fixed-point Iteration}
\label{algo: FFPI}
\begin{algorithmic}[1]
\Input
Matrix collection $\bA$ satisfying Assumption~\ref{assumption 1} and initial matrix $R_0 \in \GL_n(\R)$.
\For
{$t = 0,1,2,\dots$}
    \State
    Compute a decomposition
    \[
    \sum_{i=1}^k A_{i}^{} R_t^\top R_t^{} A_{i}^\top = C_t^{} C_t^\top
    \]
    with $C_t \in \GL_m(\R)$ and set
    \[
    L_{t+1} = \frac{1}{\sqrt{m}} C_t^{-1}.
    \]
    \State
    Compute a decomposition
    \[
    \sum_{i=1}^k A_{i}^{\top} L_{t+1}^\top L_{t+1}^{} A_{i}^{} = D_t^{} D_t^\top,
    \]
    with $D_t \in \GL_n(\R)$ set
    \[
    R_{t+1} = \frac{1}{\sqrt{n}} D_t^{-1}.
    \]
\EndFor
\end{algorithmic}
\end{algorithm}

This algorithm is also well-defined under Assumption~\ref{assumption 1}. It implicitly generates another sequence
\[
 \bA_t^{} = L_t^{} \bA R_t^\top.
\]
of scaled matrices $(A_{t,1},\dots,A_{t,k})$. It turns out that up to possible orthogonal transformations, it is the same as the sequence $\bar{\bA}_t$ generated in Algorithm~\ref{algo: OSI}. In this sense, both versions of the algorithms are in fact equivalent.

\begin{proposition}\label{prop: equivalence of algorithms}
Given Assumption~\ref{assumption 1}, let $\bar{\bA}_t$ be the sequence generated recursively in the operator Sinkhorn iteration (Algorithm~\ref{algo: OSI}), and let $\bA_t = L_t \bA R_t$ be sequences of matrix tuples generated by the fixed-point iteration (Algorithm~\ref{algo: FFPI}) with initial matrix $R_0 = I_m$, respectively. Then there exist sequences of orthogonal matrices $P_t$ and $Q_t$ with $P_0 = I$ and $Q_0 = I$ such that
\[
 \bA_t^{} = P_t^{} \bar{\bA}_t^{} Q_t^\top.
\]
\end{proposition}

The proof is by induction and is presented for completeness in Appendix~\ref{sec: proof proposition}. Observe that the operator scaling problem is itself subject to an orthogonal invariance: if $\bar{\bA}$ solves~\eqref{eq: OSa}--\eqref{eq: OSb}, then so does $Q \bar{\bA} P^\top$ for any orthogonal $Q$ and $P$. Likewise, simply applying an orthogonal scaling $P \bA Q^\top$ to any ``approximate'' solution $\bA$ will not change the distance to a true solution in any reasonable orthogonally invariant error measure. For example, we have
\[
 \left\| \sum_{i=1}^k P A_i Q^\top (PA_iQ^\top)^\top - \frac{1}{m} I \right\| = \left\| P \left( \sum_{i=1}^k A_i A_i^\top - \frac{1}{m} I \right) P^\top \right\| = \left\| \sum_{i=1}^k A_i A_i^\top - \frac{1}{m} I \right\|
\]
in any unitarily invariant matrix norm. In this light, Proposition~\ref{prop: equivalence of algorithms} allows us to focus our attention in the following entirely on Algorithm~\ref{algo: FFPI}, which appears easier to investigate.

Both Algorithm~\ref{algo: OSI} and~\ref{algo: FFPI} still allow for several different realizations, since it has not been specified how to choose the factors $\bar C_t$ and $\bar D_t$ (resp.~$C_t$ and $D_t$) in Algorithm~\ref{algo: OSI} (resp.~Algorithm~\ref{algo: FFPI})  in the factorization steps. From the computational perspective, one should choose the most efficient factorizations, e.g., based on the Cholesky decomposition. However, from the theoretical perspective, this non-uniqueness poses a potential obstacle to the analysis of the algorithms and should desirably be removed.

We focus on Algorithm~\ref{algo: FFPI}. A very natural way to overcome the issue of non-uniqueness is to investigate the sequences
\[
 X_t = L_t^\top L_t^{} \quad \text{and} \quad  Y_t = R_t^\top R_t^{}
\]
of symmetric positive definite (PD) matrices. Different from the scaling factors $(L_t,R_t)$ themselves, they are independent from the particular choice of the decomposition factors $C_t$ and $D_t$. This observation allows to interpret Algorithm~\ref{algo: FFPI} as an alternating fixed-point iteration between the two cones
\[
 \caC_m = \{ X \in \R^{m \times m}_{\mathrm{sym}} : \text{$X$ is a PD matrix} \} \quad \text{and} \quad \caC_n = \{ Y \in \R^{n \times n}_{\mathrm{sym}} : \text{$Y$ is a PD matrix} \},
\]
where and in the following $\R^{m \times m}_{\mathrm{sym}}$ denotes the linear space of real symmetric $m \times m$ matrices. To formulate this fixed-point iteration, consider the linear operators
\[
 \Phi : \R^{n \times n}_{\mathrm{sym}} \to \R^{m \times m}_{\mathrm{sym}}, \quad \Phi(Y) = \sum_{i=1}^k A_i Y A_i^\top
\]
and
\[
 \Phi^* : \R^{m \times m}_{\mathrm{sym}} \to \R^{n \times n}_{\mathrm{sym}}, \quad \Phi^* (X) = \sum_{i=1}^k A_i^\top X A_i^{}
\]
induced by the fixed collection $\bA = (A_1,\dots,A_k)$. In operator scaling, such a linear map $\Phi$ is called a \emph{completely positive (CP) map}, the matrices $A_i$ are called the \emph{Kraus operators}, and $\Phi^*$ is called the \emph{dual} of~$\Phi$.
Note that under Assumption~\ref{assumption 1}, we have
\[
\Phi(\caC_n) \subseteq \caC_m \quad \text{and} \quad \Phi^*(\caC_m) \subseteq \caC_n 
\]
(cf.~Lemma~\ref{lem: assumption 1}). We also introduce the nonlinear operators
\begin{equation}\label{eq: operators S1 and S2}
S_1(Y) \coloneqq \frac{1}{m} \Phi(Y)^{-1} \quad \text{and} \quad S_2(X) \coloneqq \frac{1}{n} \Phi^*(X)^{-1},
\end{equation}
which define maps from $\caC_n$ to $\caC_m$ (resp.~$\caC_m$ to $\caC_n$). The alternating fixed-point iteration realized by Algorithm~\ref{algo: FFPI} can then be written in the way presented in Algorithm~\ref{algo: PD-FPI}.

\begin{algorithm}[t]
\small
\caption{Fixed-Point Iteration on PD Cones}
\label{algo: PD-FPI}
\begin{algorithmic}[1]
\Input
Matrix collection $\bA$ satisfying Assumption~\ref{assumption 1} and initial matrix $Y_0 \in \caC_n$.
\For
{$t = 0,1,2,\dots$}
    \State
    Compute
    \[
    X_{t+1} = S_1(Y_t) = \frac{1}{m} \left(\sum_{i=1}^k A_i Y_t A_i^\top \right)^{-1}.  
    \]
    \State
    Compute
    \[
    Y_{t+1} = S_2(X_{t+1}) = \frac{1}{n} \left(\sum_{i=1}^k A_i^\top X_{t+1} A_i \right)^{-1}.
    \]
\EndFor
\end{algorithmic}
\end{algorithm}

We can formulate Algorithm~\ref{algo: PD-FPI} in compact form as a fixed-point iteration
\begin{equation}\label{eq: full fpi}
 \begin{pmatrix} X_{t+1} \\ Y_{t+1} \end{pmatrix} = \bS\!\begin{pmatrix} X_t \\ Y_t \end{pmatrix}
\end{equation}
with the mapping $\bS : \caC_m \times \caC_n \to \caC_m \times \caC_n$ defined as
\begin{equation}\label{eq: fixed-point map}
\bS\!\begin{pmatrix} X \\ Y \end{pmatrix} = \begin{pmatrix} S_1(Y) \\ S_2(S_1(Y)) \end{pmatrix}.
\end{equation}
Any fixed point $(X_*,Y_*) \in \caC_m \times \caC_n$ of this iteration provides solutions $(L_*,R_*)$ to the operator scaling problem~\eqref{eq: OSa}--\eqref{eq: OSb} via decomposition
\[
 X_* = L_*^\top L_*^{}, \qquad Y_* = R_*^\top R_*^{}.
\]

Observe that $(X_*,Y_*)$ is a fixed point of $\bS$ if and only if $(\mu X_*, \mu^{-1} Y_*)$ is a fixed point for any $\mu > 0$. The fixed-point iteration reflects this scaling indeterminacy in the sense that 
\begin{equation}\label{eq: orbital equivariance}
\bS\!\begin{pmatrix} \mu X_t \\ \mu^{-1} Y_t \end{pmatrix} = \begin{pmatrix} \mu X_{t+1} \\ \mu^{-1} Y_{t+1} \end{pmatrix},
\end{equation}
that is, the map $\bS$ is equivariant under the corresponding group action. We will call the trajectory of $(\mu X, \mu^{-1} Y)$ for $\mu > 0$ the \emph{orbit} of a pair $(X,Y)$. Then~\eqref{eq: orbital equivariance} shows that we can in fact interpret~\eqref{eq: full fpi} (and hence Algorithm~\ref{algo: PD-FPI}) as a fixed-point iteration on the quotient space of orbits.

Note that in practice the usually most efficient and stable choice for realizing Algorithm~\ref{algo: PD-FPI} is still by executing Algorithm~\ref{algo: FFPI} using the standard Cholesky decomposition. However, for the theoretical investigation, we will focus on Algorithm~\ref{algo: PD-FPI} from now on.

\subsection{Recap on global convergence}

As has been outlined in previous works~\cite{Georgiou2015,Idel2016}, under some additional assumptions, the fixed-point formulation on PD cones provides an elegant way of proving the global convergence of the Sinkhorn algorithm to a solution of the operator scaling problem via nonlinear Perron-Frobenius theory. We briefly recap this approach here as it will be useful for defining overrelaxation based on geodesics in Section~\ref{sec: overrelaxation}. 

First recall the \emph{Hilbert metric} on a cone of real PD matrices: setting
\begin{align*}
 M(X, \tilde X) &= \inf\{\beta \in \R : X \preceq \beta \tilde X\} = \lambda_{\max}(\tilde X^{-1/2} X \tilde X^{-1/2}), \\
 m(X, \tilde X) &= \sup\{\alpha \in \R :  \alpha \tilde X  \preceq X \} = \lambda_{\min}(\tilde X^{-1/2} X \tilde X^{-1/2})
\end{align*}
for two PD matrices $X, \tilde X$ (where $\preceq$ is the Loewner partial ordering), the Hilbert metric is defined as 
\[
 d_H(X, \tilde X) = \log \left(\frac{M(X, \tilde X)}{m(X, \tilde X)}\right) = \log \kappa(\tilde X^{-1/2} X \tilde X^{-1/2})
\]
where $\kappa$ is the spectral condition number. Precisely speaking, this defines a metric on the space of one-dimensional rays in the positive definite cone since $d_H(X, \tilde X)$ is invariant under scalar multiplication of the arguments. Equivalently, it is a metric on the convex set of PD matrices with trace one (density matrices); see, e.g.,~\cite[Chapter~2]{Lemmens2012}. Remarkably, this set then even becomes a complete metric space when equipped with the Hilbert metric~\cite[Proposition~2.5.4]{Lemmens2012}.

One can now investigate the contraction properties of the maps $S_1$ and $S_2$ defining the half-steps in Algorithm~\ref{algo: PD-FPI} with respect to the Hilbert metric. Since for any positive definite $X$ and $\tilde X$ one has $d_H(X,\tilde X) = d_H(X^{-1}, {\tilde X}^{-1})$, and using the scaling invariance, one observes
\[
 d_H(S_1(Y),S_1(\tilde Y)) = d_H(\Phi(Y),\Phi(\tilde Y)).
\]
The well-known Birkhoff--Hopf theorem~\cite[Theorem~3.5]{Eveson1995} for cone preserving linear maps bounds the contraction ratio
\begin{equation}\label{eq: contraction ratio}
 \sup_{Y, \tilde Y \in \caC_n} \left( \frac{d_H(\Phi(Y),\Phi(\tilde Y))}{d_H(Y,\tilde Y)} \right) = \Lambda_1 \coloneqq \frac{\chi - 1}{\chi + 1} < 1
\end{equation}
with
\[
 \chi^2 = \sup_{Y,\tilde Y \in \caC_n} \left( \frac{M(\Phi(Y),\Phi(\tilde Y))}{m(\Phi(Y),\Phi(\tilde Y))} \right) = \sup_{Y,\tilde Y \in \caC_n} \exp [ d_H(\Phi(Y), \Phi(\tilde Y))],
\]
provided that $\chi$ is finite. To ensure the latter, one can make additional assumptions. In~\cite{Georgiou2015}, the following property has been considered.

\begin{definition}\label{def:positivity improving}
A CP map $\Phi$ is said to be \emph{positivity improving} if it maps nonzero positive semidefinite (PSD) matrices to positive definite ones.
\end{definition}

Note that in this definition it would suffice to require that $\Phi$ maps rank-one positive semidefinite matrices to positive definite matrices. Based on this one can show that $\Phi$ is positivity improving if and only if for any vectors $v, w$ the equations $w^\top A_i v = 0$ for $i=1,\dots,k$ imply $v = 0$ or $w = 0$. Clearly, $k \ge \max(m,n)$ is a necessary condition for this property to hold. This characterization also implies that $\Phi$ is positivity improving if and only if $\Phi^*$ is positivity improving. 

If $\Phi$ is positivity improving, $\chi^2$ is indeed finite, since one then has, noting the scaling invariance of the Hilbert metric,
\[
 \chi^2 \le \sup_{\substack{Y, \tilde Y \succeq 0 \\ \tr(Y) = \tr(\tilde Y) = 1 }} \exp [d_H(\Phi(Y), \Phi(\tilde Y))]
\]
($\tr$ denoting the trace), where the right hand side is the supremum of a continuous function over a compact set and hence finite (in fact, $\chi^2$ equals this supremum). By~\eqref{eq: contraction ratio}, the map $S_1$ then is a contraction in Hilbert metric with Lipschitz constant $\Lambda_1 < 1$. Similarly, $S_2$ is shown to be a contraction in Hilbert metric with a Lipschitz constant $\Lambda_2 < 1$. We note that in the case of matrix scaling, the corresponding Lipschitz constants $\Lambda_1$ and $\Lambda_2$ are equal (which follows from an explicit formula; see,~e.g.,~\cite[Theorem~A.6.2]{Lemmens2012}), but we were not able to clarify whether this is true or not in our more general case of operator scaling.

As a result of these considerations, one can prove the following existence and convergence statement.

\begin{theorem}[{cf.~\cite[Theorem~6]{Georgiou2015} and~\cite[Corollary~9.16]{Idel2016}}]\label{thm: global convergence fpi}
Assume that the CP map $\Phi$ with Kraus operators $\bA = (A_1,\dots,A_k)$ is positivity improving. Then the map $\bS$ in~\eqref{eq: fixed-point map} has a unique orbit $(\mu X_*,\mu^{-1} Y_*)$, $\mu > 0$, of fixed points in $\caC_m \times \caC_n$. For any initialization $Y_0 \in \caC_n$, the iterates $(X_t,Y_t)$ of Algorithm~\ref{algo: PD-FPI} converge to $(X_*,Y_*)$ in the Hilbert metric (i.e.~in the sense of orbits) at a linear rate satisfying
\[
 d_H(X_{t+1},X_*) \le (\Lambda_1 \Lambda_2)^t \Lambda_1 d(Y_0,Y_*) \quad \text{and} \quad d_H(Y_{t+1},Y_*) \le  (\Lambda_1 \Lambda_2)^{t+1} d(Y_0,Y_*)
\]
for all $t \ge 0$, where $0 < \Lambda_1, \Lambda_2 < 1$ are Lipschitz constants of $S_1$ and $S_2$ with respect to the Hilbert metric.
\end{theorem}

\begin{proof}
We only sketch the proof. Let $\Sigma_m = \{ X \in \caC_m \colon \tr(X) = 1\}$. Consider the map $\hat \bS : \Sigma_m \times \Sigma_n \to \Sigma_m \times \Sigma_n$ given by applying $\bS$ and then normalizing both components to trace one. Using the invariance of the Hilbert metric under scaling, it is not difficult to show that $\hat \bS$ is a contraction with respect to the metric $\max(d_H(X,\tilde X), d(Y,\tilde Y))$ with Lipschitz constant bounded by $\Lambda_1$. Since $\Sigma_m \times \Sigma_n$ is a complete metric space under this metric, the existence of a unique fixed point $(\hat X_*, \hat Y_*)$ of $\hat \bS$ follows from the Banach fixed-point theorem. The fixed point satisfies $S_1(\hat Y_*) = \alpha \hat X_*$ and $S_2(\hat X_*) = \beta \hat Y_*$ for some $\alpha,\beta > 0$. Using the definition of $S_1$ and $S_2$ these relations can be written as
\[
 \sum_{i=1}^k \hat X_*^{1/2} A_i^{} \hat Y_* A_i^\top \hat X_*^{1/2} = \frac{1}{m \alpha} I_m, \qquad
 \sum_{i=1}^k \hat Y_*^{1/2} A_i^\top \hat X_* A_i^\top \hat Y_*^{1/2} = \frac{1}{n \beta} I_n.
\]
Noting that the traces of the left sides in both equations are equal implies $\alpha = \beta$. It is then easy to conclude that $(X_*,Y_*) = (\sqrt{\alpha} \hat X_*, \sqrt{\alpha} \hat Y_*)$ is a fixed point of $\bS$. At the same time, the uniqueness of $(\hat X_*, \hat Y_*)$ implies that $(\mu X_*, \mu^{-1} Y_*)$ with $\mu > 0$ is the only orbit of fixed points of~$\bS$. The convergence of the $(X_t,Y_t)$ to $(X_*,Y_*)$ in the Hilbert metric at the asserted rates then follows directly from $X_{t+1} = S_1(S_2(X_t))$ for $t \ge 1$ and $Y_{t+1} = S_2(S_1(Y_t))$ for $t \ge 0$ using the Lipschitz constants of $S_1$ and $S_2$.
\end{proof}

One should note that $\Phi$ being positivity improving is a convenient but not necessary condition for the solvability of the operator scaling problem~\eqref{eq: OSa}--\eqref{eq: OSb}. In~\cite{Gurvits2004} the existence of solutions is shown for the broader class of so called indecomposable CP maps by proving the convergence of the scaling algorithm in its original form of Algorithm~\ref{algo: OSI}, albeit without a rate. Moreover, solutions are also unique (up to orbits) in this case. More generally, a CP map has a solution to operator scaling if and only if it is the direct sum of indecomposable CP maps~\cite[Theorem~9.4]{Idel2016}.
This analysis is however much more involved. Anyway, in practice, the algorithm works quite reliably without checking assumptions on the CP map, for example also in cases when $k < \max(m,n)$ (not included in the experiment section), which cannot be positivity improving.

\subsection{Interpretation as alternating optimization method}\label{sec: alternating optimization}

It is often useful to interpret operator scaling from a viewpoint of optimization. Consider the function
\begin{equation}\label{eq: cost function}
\begin{aligned}
 f(X,Y) &= \tr (X \Phi(Y) ) -\frac{1}{m} \log \det (X) - \frac{1}{n} \log \det (Y) \\
&= \tr (\Phi^*(X) Y ) -\frac{1}{m} \log \det (X) - \frac{1}{n} \log \det (Y),
\end{aligned}
\end{equation}
which is defined on $\caC_m \times \caC_n \subseteq \R^{m \times m}_{\mathrm{sym}} \times \R^{n \times n}_{\mathrm{sym}}$. Using that the gradient of the log-determinant at positive definite $X$ is given by $X^{-1}$ (see,~e.g.,~\cite[Section~A.4.1]{Boyd2004}) it holds that
\[
 \nabla_X f(X,Y) = \Phi(Y) - \frac{1}{m} X^{-1} \quad \text{and} \quad \nabla_Y f(X,Y) = \Phi^*(X) - \frac{1}{n} Y^{-1}. 
\]
Comparing with~\eqref{eq: operators S1 and S2} we see that the single steps in Algorithm~\ref{algo: PD-FPI} alternatingly set one of the partial gradients of $f$ to zero. On the other hand, the partial Hessians $\nabla_{XX} f(X,Y)$ and $\nabla_{YY} f(X,Y)$, given by the linear maps
\begin{equation}\label{eq: partial Hessians}
 H_1 \mapsto \nabla_{XX} f(X,Y)[H_1] = \frac{1}{m} X^{-1} H_1 X^{-1} \quad \text{and} \quad H_2 \mapsto \nabla_{YY} f(X,Y)[H_2] = \frac{1}{n} Y^{-1} H_2 Y^{-1}
\end{equation}
on $\R^{m \times m}_{\mathrm{sym}}$ and $\R^{n \times n}_{\mathrm{sym}}$, respectively, are positive definite (with respect to Frobenius inner products) linear maps for all $(X,Y) \in \caC_m \times \caC_n$, which shows that the restricted functions $X \mapsto f(X,Y)$ and $Y \mapsto f(X,Y)$ (with one variable being fixed) are actually strictly convex. Therefore, the single steps in the algorithm actually minimize $f$ with respect to one variable, that is
\begin{align*}
 X_{t+1} &= \argmin_{X \in \caC_m} f(X, Y_t), \\
 Y_{t+1} &= \argmin_{Y \in \caC_n} f(X_{t+1}, Y).
\end{align*}
Therefore, Algorithm~\ref{algo: PD-FPI} can be interpreted as an alternating optimization method for the function~$f$. 

The solutions of the operator scaling problem are the stationary points of $f$. Note that since $f$ is constant on the nonconvex orbits $(\mu X, \mu^{-1} Y)$, it cannot be jointly-convex in  $(X,Y)$ with respect to the Euclidean geometry. However, one can show that $f$ is a \emph{geodesically convex} function~(see appendix~\ref{sec:g-convex-f}), with geodesics defined block-wise as in Section~\ref{sec: overrelaxation Hilbert metric}. This implies that in operator scaling we are indeed looking for a \emph{global} minimum $(X_*,Y_*)$ of the function $f$.

Under the additional assumption that the linear map $\Phi$ is positivity improving we have already seen in Theorem~\ref{thm: global convergence fpi} that a stationary point $(X_*,Y_*)$ is unique up to its one-dimensional orbit. Indeed, in this case we can show that the Hessian of $f$ at $(X_*,Y_*)$ has only a one-dimensional null space tangential to the orbit (Morse--Bott property). This property will be needed in the local convergence analysis in Section~\ref{sec: local convergence}. 

\begin{lemma}\label{lem: null space of Hessian}
Assume that the CP map $\Phi$ is positivity improving, and let $(X_*,Y_*)$ be a fixed-point of the map $\bS$ in~\eqref{eq: fixed-point map}. Then the Hessian $\nabla^2 f(X_*,Y_*)$ is positive semidefinite with a one-dimensional null space spanned by $(X_*,-Y_*)$. In other words, the null space of $\nabla^2 f(X_*,Y_*)$ equals the tangent space to the orbit $(\mu X_*,\mu^{-1} Y_*)$ at $\mu = 1$.
\end{lemma}

\begin{proof}
According to the considerations preceding the lemma, $(X_*,Y_*)$ is a global minimum of $f$. Therefore, the Hessian is positive semidefinite. Since $\nabla f$ is constant zero on the orbit $(\mu X_*, \mu^{-1} Y_*)$, it also follows that the tangent space to the orbit at $(X_*,Y_*)$ is in the null space of the Hessian. This can also be verified directly: as a linear map on $\R^{m \times m}_{\mathrm{sym}} \times \R^{n \times n}_{\mathrm{sym}}$ the Hessian at $(X_*,Y_*)$ reads
\[
\begin{pmatrix}
 H_1 \\ H_2
\end{pmatrix}
\mapsto
\begin{pmatrix}
 \Phi(H_2) + \frac{1}{m} X_*^{-1} H_1 X_*^{-1} \\ \Phi^*(H_1) + \frac{1}{n} Y_*^{-1} H_2 Y_*^{-1}
\end{pmatrix}.
\]
Using the fixed-point property $\Phi(Y_*) = \frac{1}{m} X_*^{-1}$ and $\Phi^*(X_*) = \frac{1}{n} Y_*^{-1}$ we see that $(H_1,H_2) = (X_*,-Y_*)$ is indeed in the null space of the Hessian. We will show that this is the only direction in the null space.

As can be derived from the above explicit expression of the Hessian, a necessary condition for $(H_1,H_2)$ to be in the null space is that the matrix $\tilde H_1 = X_*^{-1/2} H_1 X_*^{-1/2}$ satisfies the equation
\[
 \tilde H_1 = mn X_*^{1/2} \Phi \big( Y_* \Phi^*(X_*^{1/2} \tilde H_1 X_*^{1/2}) Y_* \big) X_*^{1/2} \eqqcolon \Gamma( \tilde H_1 ),
\]
that is, $\tilde H_1$ must be an eigenvector of the linear operator $\Gamma$ on $\R^{m \times m}_{\mathrm{sym}}$ with eigenvalue one. We know that $\tilde H_1 = I_m$ (corresponding to $H_1 = X_*$) is such an eigenvector, and the lemma will be proved if we show that (up to scalar multiplication) it is the only one. To do this, let
\begin{equation}\label{eq:expansion H_1}
 \tilde H_1 = \sum_{i=1}^m z_i^{} u_i^{} u_i^\top
\end{equation}
be a spectral decomposition with pairwise orthonormal vectors $u_1,\dots,u_m \in \R^m$ and scalars $z_1,\dots,z_m$, so that
\[
 \Gamma(\tilde H_1) = \sum_{i=1}^m z_i^{} \Gamma(u_i^{} u_i^\top).
\]
If $\tilde H_1 = \Gamma(\tilde H_1)$, then by comparing coefficients for the $u_1^{} u_1^\top,\dots,u_m^{} u_m^\top$ we obtain a relation
\[
 z = B z
\]
where the matrix $B = [b_{ij}]$ has entries
\[
 b_{ij} = \tr(\Gamma(u_j^{} u_j^\top) u_i^{} u_i^\top) = u_i^\top \Gamma(u_j^{} u_j^\top) u_i.
\]
Although this matrix depends on the orthonormal basis $u_1,\dots,u_m$, the vector $z = \mathbf 1$ always satisfies the above relation, that is, $B \mathbf 1 = \mathbf 1$, since $I_m = \sum_{i=1}^m u_i^{} u_i^\top$ for any choice of basis. The main observation now is that since $\Phi^*$ (together with $\Phi$) is positivity improving, $\Gamma(u_j^{} u_j^\top)$ is a positive definite matrix for every~$j$, implying that all entries of $B$ are strictly positive. We can then deduce from the Perron--Frobenius theorem that $\mathbf 1$ is the only eigenvector of $B$ with eigenvalue one. This shows that for any orthonormal basis $u_1,\dots,u_m$, $\tilde H_1 = I_m$ is the only eigenvector of $\Gamma$ of the form~\eqref{eq:expansion H_1} with eigenvalue one, which concludes the proof.
\end{proof}

\section{Overrelaxation and accelerated convergence}\label{sec: overrelaxation}

Our goal is to accelerate the convergence of the fixed-point iteration in Algorithm~\ref{algo: PD-FPI} by means of overrelaxation. We discuss several versions, which however we show are equivalent in the first order near fixed points, and in particular yield the same asymptotic local convergence rates.

\subsection{Overrelaxation in Euclidean metric}\label{sec: overrelaxation Euclidean}

A heuristic approach to overrelaxation would be to replace the fixed-point iteration~\eqref{eq: full fpi} with
\begin{equation}\label{eq: overrelaxation Euclidean}
 \begin{pmatrix} X_{t+1} \\ Y_{t+1} \end{pmatrix} =  \begin{pmatrix} (1-\omega) X_t + \omega S_1(Y_t) \\ (1-\omega) Y_t + \omega S_2(X_{t+1}) \end{pmatrix},
\end{equation}
where $\omega \in \mathbb R$ is the relaxation parameter. This results in Algorithm~\ref{algo: overrelaxation PD}. As we will see in the experiments, this can work well in practice, although it is not quite rigorous from a theoretical point of view: when $\omega > 1$ is used (which is the case of interest in overrelaxation) it cannot be guaranteed that positive definite matrices are generated. However, if $(X_*,Y_*) \in \caC_m \times \caC_n$ is a fixed point of $\bS$, then at least for some neighborhood $\mathcal O \subset \caC_n$ (which depends on $\omega$) of $Y_*$ the corresponding iteration map
\[
 \bS_\omega\!\begin{pmatrix} X \\ Y \end{pmatrix} =  \begin{pmatrix} (1-\omega) X + \omega S_1(Y) \\ (1-\omega) Y + \omega S_2((1-\omega) X + \omega S_1(Y)) \end{pmatrix}
\]
is well defined as a map from $\caC_m \times \mathcal O$ to $\caC_m \times \caC_n$ and allows to write a single step of~\eqref{eq: overrelaxation Euclidean} as
\[
 \begin{pmatrix} X_{t+1} \\ Y_{t+1} \end{pmatrix} = \bS_\omega\!\begin{pmatrix} X_{t} \\ Y_{t} \end{pmatrix},
\]
provided $Y_t \in \mathcal O$. Therefore, $\bS_\omega$ is in fact well defined in some neighborhood of the whole orbit $(\mu X_*, \mu^{-1} Y_*)$. Obviously $\bS_\omega$ and $\bS$ have the same fixed points. Note that~$\bS_\omega$ also shares the orbital equivariance property, that is,
\begin{equation}\label{eq: orbital equivariance relaxation}
 \bS_\omega\!\begin{pmatrix} \mu X_{t} \\ \mu^{-1}Y_{t} \end{pmatrix} = \begin{pmatrix} \mu X_{t+1} \\ \mu^{-1}Y_{t+1} \end{pmatrix},
\end{equation}
and can hence be regarded as a map between orbits.

\begin{algorithm}[t]
\small
\caption{Fixed-Point Iteration on PD cone with Overrelaxation (PD-OR)}\label{algo: overrelaxation PD}
\begin{algorithmic}[1]
\Input
Matrix collection $\bA$ satisfying Assumption~\ref{assumption 1},  initial matrices $(X_0, Y_0) \in \caC_m \times \caC_n$, and relaxation parameter $\omega >0$.
\For
{$t = 0,1,2,\dots$}
    \State
    Compute
    \[
    S_1(Y_t) = \frac{1}{m} \left(\sum_{i=1}^k A_i Y_t A_i^\top \right)^{-1}.  
    \]
    and set
    \[
    X_{t+1} = (1-\omega) X_t + \omega S_1(Y_t).
    \]
    Let $L_{t+1}$ be the Cholesky factor of $X_{t+1}$.
    \State
    Compute
    \[
    S_2(X_{t+1}) = \frac{1}{n} \left(\sum_{i=1}^k A_i^\top X_{t+1} A_i \right)^{-1}. 
    \]
    and set
    \[
    Y_{t+1} = (1-\omega) Y_t + \omega S_2(X_{t+1}).
    \]
    Let $R_{t+1}$ be the Cholesky factor of $Y_{t+1}$.
\EndFor
\end{algorithmic}
\end{algorithm}

\subsection{Overrelaxation with coordinate transformation}

A possible approach for mitigating the problem that linear combinations of PD matrices are not necessarily PD, is to apply a suitable coordinate transformation
\[
 \begin{pmatrix} \hat X \\ \hat Y \end{pmatrix} = \begin{pmatrix} \psi_1(X) \\ \psi_2(Y) \end{pmatrix} \eqqcolon \Psi\!\begin{pmatrix} X \\ Y \end{pmatrix}
\]
that realizes a bijection between $\caC_m \times \caC_n$ and a \emph{linear} space. One can then perform the actual overrelaxation in the transformed coordinates, that is,
\[
 \begin{pmatrix} \hat X_{t+1} \\ \hat Y_{t+1} \end{pmatrix} =  \begin{pmatrix} (1-\omega) \hat X_t + \omega \hat S_1(\hat Y_t) \\ (1-\omega) \hat Y_t + \omega \hat S_2(\hat X_{t+1}) \end{pmatrix},
\]
where $\hat S_1 = \psi_1^{} \circ S_1 \circ \psi_2^{-1}$ and similarly for $\hat S_2$. We can write this as
\[
 \begin{pmatrix} \hat X_{t+1} \\ \hat Y_{t+1} \end{pmatrix} = \hat \bS_\omega\!\begin{pmatrix} \hat X_{t} \\ \hat Y_{t} \end{pmatrix}.
\]
In the original coordinates the iteration reads
\begin{equation}\label{eq: overrelaxation variable transform}
 \begin{pmatrix} X_{t+1} \\ Y_{t+1} \end{pmatrix} = \bS_{\omega,\Psi}\!\begin{pmatrix} X_{t} \\  Y_{t} \end{pmatrix}
\end{equation}
with $\bS_{\omega,\Psi} \coloneqq \Psi^{-1} \circ \hat \bS_\omega \circ \Psi$. However, note that ensuring the desired orbital equivariance, that is,
\begin{equation}\label{eq: orbital equivariance psi}
 \bS_{\omega,\Psi}\!\begin{pmatrix} \mu X_{t} \\  \mu^{-1} Y_{t} \end{pmatrix} = \begin{pmatrix} \mu X_{t+1} \\ \mu^{-1} Y_{t+1} \end{pmatrix}
\end{equation}
requires certain compatibility conditions, like
\[
 \psi_1^{-1}[ (1-\omega) \psi_1(\mu X) + \omega \psi_1(\mu \tilde X) ] = \mu \psi_1^{-1}[ (1-\omega) \psi_1(X) + \omega \psi_1(\tilde X) ]
\]
for all $\mu > 0$ and similar for $\psi_2$. While this poses a general restriction, it is indeed satisfied in the following two important examples.

\paragraph{Matrix logarithm.}
The somewhat most natural choice for $\psi_1$ and $\psi_2$ would be the matrix logarithm for PD matrices, that is,
\[
 \begin{pmatrix} \hat X \\ \hat Y \end{pmatrix} = \begin{pmatrix} \log X \\ \log Y \end{pmatrix},
\]
the inverse transformations being the matrix exponential. Log-coordinates are known to be extremely useful in the special case of matrix scaling, where $X$ and $Y$ can be taken diagonal, and corresponding overrelaxation schemes for the Sinkhorn algorithm in log-space have been studied in~\cite{Thibault2021,Lehmann2022}. However, in our more general scenario of operator scaling the matrix logarithm (and exponential) might have a potential disadvantage of being rather expensive to compute accurately in practice. Yet it is worth mentioning that the orbital equivariance property~\eqref{eq: orbital equivariance psi} of the corresponding overrelaxation scheme in log coordinates would be satisfied.

\paragraph{Cholesky decomposition.}
A cheaper and reasonable choice in our context is to take factors $L = \psi_1(X)$ and $R = \psi_2(Y)$ in the factorizations $X = L^\top L$ and $Y = R^\top R$, which have to be computed anyway. When using Cholesky decomposition (of $X^{-1}$ and $Y^{-1}$), the matrices $L$ and $R$ are upper triangular and hence indeed belong to linear spaces of the same dimension as the PD cones. Technically speaking, the Cholesky factors of PD matrices matrix only constitute local diffeomorphisms, because of the additional constraint that diagonal elements should be nonzero, and uniqueness only holds up to signs of columns. These restrictions however do not pose a danger in numerical computations and the resulting overrelaxation scheme is efficient and robust. It is therefore noted as Algorithm~\ref{algo: overrelaxation Cholesky}. Moreover, the local convergence analysis as presented in Section~\ref{sec: local convergence} requires only a local diffeomorphism. The orbital equivariance~\eqref{eq: orbital equivariance psi} is ensured as well.

\begin{algorithm}[t]
\small
\caption{Cholesky Factor Fixed-Point Iteration with Overrelaxation (Cholesky-OR)}\label{algo: overrelaxation Cholesky}
\begin{algorithmic}[1]
\Input
Matrix collection $\bA$ satisfying Assumption~\ref{assumption 1},  initial matrices $(L_0, R_0) \in \GL_m(\R) \times \GL_n(\R)$, and relaxation parameter $\omega >0$.
\For
{$t = 0,1,2,\dots$}
    \State
    Compute a Cholesky decomposition
    \[
    \sum_{i=1}^k A_{i}^{} R_t^\top R_t^{} A_{i}^\top = C_t^{} C_t^\top
    \]
    with $C_t$ lower triangular and set
    \[
    L_{t+1} = (1-\omega) L_t + \frac{\omega}{\sqrt{m}} C_t^{-1}.
    \]
    \State
    Compute a Cholesky decomposition
    \[
    \sum_{i=1}^k A_{i}^{\top} L_{t+1}^\top L_{t+1}^{} A_{i}^{} = D_t^{} D_t^\top,
    \]
    with $D_t$ lower triangular and set
    \[
    R_{t+1} = (1-\omega ) R_t + \frac{\omega}{\sqrt{n}} D_t^{-1}.
    \]
\EndFor
\end{algorithmic}
\end{algorithm}

\subsection{Geodesic overrelaxation}\label{sec: overrelaxation Hilbert metric}

We next derive a geometric version of overrelaxation which takes the hyperbolic geometry of the PD cones induced by the Hilbert metric into account. For two PD matrices $X$ and $\tilde X$ consider the operator
\begin{align}\label{eq:sharp}
 X \#_\omega \tilde X \coloneqq X^{1/2}(X^{-1/2}\tilde X X^{-1/2})^\omega X^{1/2},
\end{align}
with a parameter $\omega \in \R$. Note that for $\omega \in [0,1]$ the curve $\omega \mapsto X \#_\omega \tilde X$ connects the matrices $X$ and~$\tilde X$. Indeed, it is well-known that in the set of density matrices, the corresponding normalized curve $\omega \mapsto (X \#_\omega \tilde X) / \tr(X \#_\omega \tilde X)$ with $\omega \in [0,1]$ provides the \emph{geodesic} with respect to the Hilbert metric; see, e.g., \cite[Proposition~2.6.8]{Lemmens2012}. When identifying PD matrices up to scalar multiplication, the same remains true for the operator $X \#_\omega \tilde X$ itself. This provides us with a way of defining overrelaxation for the PD-FPI iteration (Algorithm~\ref{algo: PD-FPI}) along geodesics as follows:
\begin{equation}\label{eq: overrelaxation Hilbert metric}
 \begin{pmatrix} X_{t+1} \\ Y_{t+1} \end{pmatrix} = \bS_{\#_\omega}\!\begin{pmatrix} X_t \\ Y_t \end{pmatrix} \coloneqq \begin{pmatrix} X_t \#_\omega S_1(Y_t) \\ Y_t \#_\omega S_2(X_t \#_\omega S_1(Y_t)) \end{pmatrix}.
\end{equation}
One easily verifies orbital equivariance
\begin{equation}\label{eq: orbital equivariance geodesic}
 \bS_{\#_\omega}\!\begin{pmatrix} \mu X_t \\ \mu^{-1} Y_t \end{pmatrix} = \begin{pmatrix} \mu X_{t+1} \\ \mu^{-1} Y_{t+1} \end{pmatrix}
\end{equation}
for all $\mu > 0$, which allows to interpret~\eqref{eq: overrelaxation Hilbert metric} as a fixed-point iteration on orbits. We denote the resulting scheme as Algorithm~\ref{algo: overrelaxation geodesic}. As before, we note that in practice $S_1(Y_t)$ and $S_2(X_{t+1})$ are usually best computed via Cholesky decomposition. 

\begin{algorithm}[t]
\small
\caption{Fixed-point Iteration on PD Cones with Geodesic Overrelaxation (Geodesic-OR)}
\label{algo: overrelaxation geodesic}
\begin{algorithmic}[1]
\Input
Matrix collection $\bA$ satisfying Assumption~\ref{assumption 1},  initial matrices $(X_0, Y_0) \in \caC_m \times \caC_n$, and relaxation parameter $\omega >0$.
\For
{$t = 0,1,2,\dots$}
    \State
    Compute
    \[
    S_1(Y_t) = \frac{1}{m} \left(\sum_{i=1}^k A_i Y_t A_i^\top \right)^{-1}.  
    \]
    \State
    Compute
    \[
    X_{t+1} = X_t \#_\omega S_1(Y_t)
    \]
    and let $L_{t+1}$ be the Cholesky factor of $X_{t+1}$.
    \State
    Compute
    \[
    S_2(X_{t+1}) = \frac{1}{n} \left(\sum_{i=1}^k A_i^\top X_{t+1} A_i \right)^{-1}. 
    \]
    \State
    Compute
    \[
    Y_{t+1} = Y_t \#_\omega S_2(X_{t+1})
    \]
    and let $R_{t+1}$ be the Cholesky factor of $Y_{t+1}$.
\EndFor
\end{algorithmic}
\end{algorithm}

\subsection{Local convergence}\label{sec: local convergence}

All three of the proposed overrelaxation methods, that is,~\eqref{eq: overrelaxation Euclidean},~\eqref{eq: overrelaxation variable transform} and~\eqref{eq: overrelaxation Hilbert metric}, have the same fixed points $(X_*,Y_*) \in \caC_{m} \times \caC_n$. We will show that around fixed points they are actually all equal to first order and hence have the same asymptotic linear convergence rate, which is determined by the spectral radius of the linearization in the fixed point. Since the heuristic version in~\eqref{eq: overrelaxation Euclidean} is a standard nonlinear SOR method in Euclidean space, its rate can be exactly deduced from well-known results.

Note that derivatives in this section are always understood as Fr\'echet derivatives in spaces of symmetric matrices, which is possible since $\caC_m$ and $\caC_{n}$ are open subsets in the corresponding spaces. Clearly, the derivative of $\bS_\omega$ in~\eqref{eq: overrelaxation Euclidean} in a fixed point $(X_*,Y_*)$ is given as
\[
 \bS_\omega'\!\begin{pmatrix} X_* \\ Y_* \end{pmatrix}\begin{bmatrix} H_1 \\ H_2 \end{bmatrix} = \begin{pmatrix} (1-\omega) H_1 + \omega S_1'(Y_*)[H_2] \\ (1-\omega) H_2 + \omega S_2'(X_*)[(1-\omega) H_1 + \omega S_1'(Y_*)[H_2]] \end{pmatrix}.
\]
We can write this linear operator in ``block matrix'' notation as
\begin{equation}\label{eq: derivative block notation}
 \bS_\omega'\!\begin{pmatrix} X_* \\ Y_* \end{pmatrix} = (1-\omega) \begin{pmatrix} I & 0 \\ 0 & I \end{pmatrix} + \omega  \begin{pmatrix} 0 & S_1'(Y_*) \\ (1-\omega) S_2'(X_*) & \omega S_2'(X_*) S_1'(Y_*) \end{pmatrix}.
\end{equation}

For showing that the linearization of the geodesic overrelaxation~\eqref{eq: overrelaxation Hilbert metric} in a fixed point takes the same form, we will need the derivative of the geodesic connection $(X,\tilde X) \mapsto X \#_\omega \tilde X$ at a point $X = \tilde X$.

\begin{lemma}\label{lem: derivative of geodesic}
Let $X \in \caC_{m}$ and $\omega \in \mathbb R$ be fixed. Then for symmetric $H, \tilde H \in \R^{m \times m}$ with small enough norm it holds that
\[
 (X+H)\#_{\omega} (X+\tilde H) = X + (1-\omega)H + \omega \tilde H + O(\| (H,\tilde H) \|^2),
\]
where $O(\|(H,\tilde H)\|^2)$ denotes terms (depending on $X$ and $\omega$) of at least quadratic order in $H$ and $\tilde H$.
\end{lemma}

This property appears quite intuitive from the perspective of differential geometry and might be known. We provide a direct proof using matrix analysis in Section~\ref{sec: proof lemma}.

\begin{lemma}\label{lem: equal derivatives}
Let $(X_*,Y_*)$ be a fixed-point of the map $\bS$ in~\eqref{eq: fixed-point map} and $\omega \in \mathbb R$. Then $(X_*,Y_*)$ is also a fixed point of the iterations~\eqref{eq: overrelaxation Euclidean},~\eqref{eq: overrelaxation variable transform} and~\eqref{eq: overrelaxation Hilbert metric}, with the corresponding maps $\bS_\omega$, $\bS_{\omega,\Psi}$ and $\bS_{\#_\omega}$ being well-defined and smooth in some neighborhood of $(X_*,Y_*)$ within $\caC_m \times \caC_n$. Furthermore, it holds that
\[
 \bS_\omega'\!\begin{pmatrix} X_* \\ Y_* \end{pmatrix} = \bS_{\omega,\Psi}'\!\begin{pmatrix} X_* \\ Y_* \end{pmatrix} = \bS_{\#_\omega}'\!\begin{pmatrix} X_* \\ Y_* \end{pmatrix}. 
\]
\end{lemma}

\begin{proof}
The formulas for the operators $S_1$ and $S_2$ are explicitly given in~\eqref{eq: operators S1 and S2} and show that all maps are (locally) well-defined compositions of smooth maps. The fixed-point properties are also clear by construction. An application of the chain rule to~\eqref{eq: overrelaxation variable transform} together with the inverse function theorem yields
\[
 \bS_{\omega,\Psi}'\!\begin{pmatrix} X_* \\ Y_* \end{pmatrix} = (\Psi^{-1})'\!\begin{pmatrix} \hat X_* \\ \hat Y_* \end{pmatrix} \cdot \hat \bS_\omega'\!\begin{pmatrix} \hat X_* \\ \hat Y_* \end{pmatrix} \cdot \Psi'\!\begin{pmatrix} X_* \\ Y*\end{pmatrix} = [\Psi'\!\begin{pmatrix} X_* \\ Y_* \end{pmatrix}]^{-1} \cdot \hat \bS_\omega'\!\begin{pmatrix} \hat X_* \\ \hat Y_* \end{pmatrix} \cdot \Psi'\!\begin{pmatrix} X_* \\ Y* \end{pmatrix}.
\]
However, from the structure of $\hat \bS_\omega'(\hat X_*, \hat Y_*)$ as in~\eqref{eq: derivative block notation}, but with $S_1$ and $S_2$ replaced by $\hat S_1 = \psi^{}_1 \circ S_1 \circ \psi_2^{-1}$ and $\hat S_2 = \psi_2^{} \circ S_2 \circ \psi_1^{-1}$, respectively, one can also verify that
\[
 \hat \bS_\omega'\!\begin{pmatrix} \hat X_* \\ \hat Y_* \end{pmatrix}
    = \Psi'\!\begin{pmatrix} X_* \\ Y_* \end{pmatrix} \cdot \bS_\omega'\!\begin{pmatrix} \hat X_* \\ \hat Y_* \end{pmatrix} \cdot [\Psi'\!\begin{pmatrix} X_* \\ Y_* \end{pmatrix}]^{-1}
\]
(here one exploits the structure $\Psi(X,Y) = (\psi_1(X), \psi_2(Y))$, so that $\Psi'(X_*,Y_*)$ is block diagonal). This shows $\bS_{\omega,\Psi}'(X_*,Y_*) = \bS_\omega'(X_*,Y_*)$. The equality $\bS_{\#_\omega}'(X_*,Y_*) = \bS_\omega'(X_*,Y_*)$ follows directly from applying Lemma~\ref{lem: derivative of geodesic} and the chain rule to~\eqref{eq: overrelaxation Hilbert metric} and comparing with~\eqref{eq: derivative block notation}.
\end{proof}

We can now state the local convergence result for all three versions of relaxation when $0 < \omega < 2$ is used, and under the additional assumption that $\Phi$ is positivity improving. In principle, we follow a standard reasoning based on the fact that the linearization of a nonlinear SOR iteration at a fixed point yields a linear SOR iteration for the Hessian~\cite{Ortega1970}. This is also reflected by the familiar formulas for the convergence rate in the theorem below. However, given the orbital equivariance of all three methods (as stated in~\eqref{eq: orbital equivariance relaxation},~\eqref{eq: orbital equivariance psi} and~\eqref{eq: orbital equivariance geodesic}), local convergence can only be established in the sense of orbits $(\mu X, \mu^{-1} Y)$. A possible way to do this is by passing to the quotient manifold of orbits, which would require some additional concepts. In order to bypass the related technicalities, here we will instead take a more explicit (but equivalent) viewpoint by considering the unique~\emph{balanced} representatives $(\bar X,\bar Y)$ in every orbit satisfying $\| \bar X \|_F  = \| \bar Y \|_F$ (in Frobenius norm). Note that for an arbitrary $(X,Y) \in \caC_m \times \caC_n$, the balanced representative in its orbit is
\begin{equation}\label{eq: balancing map}
 \begin{pmatrix} \bar X \\ \bar Y \end{pmatrix} = \bP\!\begin{pmatrix} X \\ Y \end{pmatrix} \coloneqq \begin{pmatrix} \frac{\| Y \|_F^{1/2}}{\| X \|_F^{1/2}} X \\ \frac{\| X \|_F^{1/2}}{\| Y \|_F^{1/2}} Y \end{pmatrix}.
\end{equation}
By proving local convergence of balanced representatives, we essentially prove local convergence of orbits.

\begin{theorem}\label{thm: local convergence}
Let $\max(m,n) > 1$. Assume that the CP map $\Phi$ is positivity improving, and let $(X_*,Y_*)$ be a fixed-point of the map $\bS$ in~\eqref{eq: fixed-point map}. For $0 < \omega < 2$, consider a sequence $(X_t,Y_t)$ generated by one of the iterations~\eqref{eq: overrelaxation Euclidean},~\eqref{eq: overrelaxation variable transform} or~\eqref{eq: overrelaxation Hilbert metric}, where for~\eqref{eq: overrelaxation variable transform} we only require that $\Psi$ is a \emph{local} diffeomorphism on some open subset containing the orbit of $(X_*, Y_*)$, and that additionally $\bS_{\omega,\Psi}$ satisfies the equivariance property~\eqref{eq: orbital equivariance psi} on that subset (both is satisfied for Algorithm~\ref{algo: overrelaxation Cholesky}).

Then if the balanced starting point $(\bar X_0,\bar Y_0)$ is close enough to the balanced fixed point $(\bar X_*,\bar Y_*)$ (how close may depend on $\omega$ and the chosen method), the sequence $(\bar X_t,\bar Y_t)$ of balanced iterates is well-defined and converges to the balanced fixed point $(\bar X_*,\bar Y_*)$ at an asymptotic rate
\[
 \limsup_{t \to \infty}  \| (\bar X_t,\bar Y_t) - (\bar X_*,\bar Y_*) \|^{1/t} \le \rho(\omega) < 1
\]
(in any norm or equivalent distance measure). Here $\rho(\omega)$ can be expressed by the formula
\[
 \rho(\omega) = \begin{cases} 1 - \omega  + \frac{1}{2} \omega^2 \beta^2 + \omega \beta \sqrt{1 - \omega + \frac14 \omega^2 \beta^2} &\quad \text{if $0 < \omega \le \omega_{\mathrm{opt}}$}, \\ \omega - 1 &\quad \text{if $\omega_{\mathrm{opt}} \le \omega < 2$}, \end{cases}
\]
where $0 \le \beta < 1$ and $\beta^2 = \rho(1)$ is the rate of the standard operator Sinkhorn iteration. The choice
\begin{equation}\label{eq:omega-opt}
 \omega_{\mathrm{opt}} = \frac{2}{1 + \sqrt{1- \beta^2}} > 1.
\end{equation}
achieves the fastest asymptotic rate
\[
 \rho(\omega_{\mathrm{opt}}) = \omega_{\mathrm{opt}} - 1 = \frac{1 - \sqrt{1-\beta^2}}{1 + \sqrt{1 - \beta^2}}.
\]
\end{theorem}

\begin{proof}
Let us first consider iterates $(X_t,Y_t)$ generated by the Euclidean version~\eqref{eq: overrelaxation Euclidean}, that is, using the map $\bS_\omega$. It is well-known, and has been worked out in detail in related work~\cite{Lehmann2022,Oseledets2023}, how the linearization of a fixed-point iteration of the form~\eqref{eq: overrelaxation Euclidean} yields a linear SOR iteration, and how to establish its local contractivity in the presence of group equivariances of Morse--Bott type. Here we adopt essentially the same reasoning for the problem at hand in order to provide a self-contained result. Given that the operators $S_1$ and $S_2$ provide the (unique) global minima for the function $f(X,Y)$ in~\eqref{eq: cost function} when one variable is fixed, we can interpret the fixed-point iteration $\bS_\omega$ as a nonlinear SOR iteration for solving the nonlinear system
\begin{align*}
 \nabla_X f(X,Y) &= 0, \\ \nabla_Y f(X,Y) &=0.
\end{align*}
Differentiating the relations $\nabla_X f(S_1(Y),Y) = 0$ and $\nabla_Y f(X,S_2(X)) = 0$ at a fixed point $(X_*,Y_*)$ and combining with~\eqref{eq: derivative block notation} allows to rewrite $\bS_\omega'(X_*,Y_*)$ as the error iteration operator
\begin{equation}\label{eq: SOR error operator}
 \bS_\omega'\!\begin{pmatrix} X_* \\ Y_* \end{pmatrix} = T_\omega \coloneqq I - N_\omega^{-1} \nabla^2 f(X_*,Y_*)
\end{equation}
of a \emph{linear} SOR iteration for the Hessian
\[
 \nabla^2 f(X_*,Y_*) = \begin{pmatrix} \nabla_{XX} f(X_*,Y_*) & \nabla_{XY} f(X_*,Y_*) \\ \nabla_{YX} f(X_*,Y_*) & \nabla_{YY} f(X_*,Y_*)   \end{pmatrix}
\]
of $f$ at $(X_*,Y_*)$, where
\[
 N_\omega = \begin{pmatrix} \frac{1}{\omega} \nabla_{XX} f(X_*,Y_*) & 0 \\ \nabla_{YX} f(X_*,Y_*) & \frac{1}{\omega} \nabla_{YY} f(X_*,Y_*) \end{pmatrix}.
\]
A more detailed derivation of the formula~\eqref{eq: SOR error operator} for $\bS_\omega'(X_*,Y_*)$ can be found in~\cite[Eqs.~(18)--(20)]{Oseledets2023}, which deals with a different function, but of the same general structure. Note that for obtaining~\eqref{eq: SOR error operator} it is required that the block diagonal of $\nabla^2 f(X_*,Y_*)$ is invertible (so that $N_\omega$ is invertible). This, however, is indeed the case since by~\eqref{eq: partial Hessians} the block diagonal is even positive definite.

We now inspect the contractivity of $T_\omega$. By Lemma~\ref{lem: null space of Hessian}, the Hessian $\nabla^2 f(X_*,Y_*)$ is positive semidefinite and its null space $V_\omega \coloneqq \ker \nabla^2 f(X_*,Y_*)$ equals the tangent space to the orbit $(\mu X_*, \mu^{-1} Y_*)$ at $\mu = 1$. Note that $V_\omega$ is an invariant subspace of $T_\omega$, on which it just acts as identity. We now invoke a version of Young's SOR theorem for positive semidefinite $2 \times 2$ block systems as presented in~\cite[Lemma~1]{Oseledets2023} (the assumption $\max(m,n) > 1$ is technically needed to apply it). It states that $T_\omega$ has another invariant subspace $W_\omega$ of co-dimension one and complementary to $V_\omega$, such that the spectral radius of the restriction of $T_\omega$ onto $W_\omega$ is strictly smaller than one and equals $\rho(\omega)$ as given by the formula in the theorem. In particular, $T_\omega$ is a contraction on $W_\omega$, but not on $V_\omega$.
    
From now on we will focus on the balanced fixed point $(\bar X_*, \bar Y_*)$ and denote by $V_\omega$ and $W_\omega$ the corresponding invariant subspaces of $\bS_\omega'(\bar X_*, \bar Y_*)$. The idea is that for the balanced sequence $(\bar X_t, \bar Y_t)$ the balancing map $\bP$ in~\eqref{eq: balancing map} removes the first order parts of the error in the tangent space $V_\omega$ because $\bP$ is constant on the orbit. Therefore, for the balanced sequence the contractivity of $\bS_\omega'(\bar X_*,\bar Y_*)$ on $W_\omega$ should be sufficient for establishing its local convergence to $(\bar X_*, \bar Y_*)$. From a related perspective of Riemannian optimization, the balanced iterates can be interpreted as representatives of the quotient manifold of orbits. The space $W_\omega$ is a \emph{horizontal space} for the quotient structure, which also indicates that only the contractivity of $\bS_\omega'(X_*,Y_*)$ on $W_\omega$ matters for the local convergence ``in the sense of orbits''. However, in the rest of the proof we will work out this idea without further reference to Riemannian optimization, but using only the balancing map $\bP$ instead.
    
By the orbital equivariance of $\bS_\omega$, the balanced iterates $(\bar X_t, \bar Y_t)$ follow the fixed-point iteration
\[
 \begin{pmatrix} \bar X_{t+1} \\ \bar Y_{t+1} \end{pmatrix} = (\bP \circ \bS_\omega)\!\begin{pmatrix} \bar X_t \\ \bar Y_t \end{pmatrix}.
\]
Since $(\bar X_*,\bar Y_*)$ is a fixed point of both $\bP$ and $\bS_\omega$, we have
\[
 (\bP \circ \bS_\omega)'\!\begin{pmatrix} \bar X_* \\ \bar Y_* \end{pmatrix} = \bP'\!\begin{pmatrix} \bar X_* \\ \bar Y_* \end{pmatrix} \cdot \bS_\omega'\!\begin{pmatrix} \bar X_* \\ \bar Y_*\end{pmatrix}.
\]
We claim that this derivative has the spectral radius equal to $\rho(\omega)$ which will imply the asserted local linear convergence at the asymptotic rate $\rho(\omega)$ by the local contraction principle. Let $\bP' = \bP'(\bar X_*,\bar Y_*)$ and $\bS_\omega' = \bS_\omega'(\bar X_*,\bar Y_*)$ for abbreviation. We will show that
\begin{equation}\label{eq: spectral radius PS}
 \lim_{t \to \infty} \| (\bP' \bS_\omega')^t \|^{1/t} = \rho(\omega).
\end{equation}
Note that by the equivariance property of $\bS_\omega$, it holds $(\bP \circ \bS_\omega) \circ \dots \circ (\bP \circ \bS_\omega) = \bP \circ (\bS_\omega \circ \dots \circ \bS_\omega)$ (where dots indicate that the composition is taken $t$ times). Hence
\[
 (\bP' \bS_\omega')^t = \bP' (\bS_\omega')^t.
\]
Let $\varepsilon > 0$. Since the restriction of $\bS_\omega'$ onto its invariant subspace $W_\omega$ has spectral radius $\rho(\omega)$ as mentioned above, it holds that the restriction of $(\bS_\omega')^t$ onto that subspace satisfies
\[
 (\rho(\omega) - \varepsilon)^t \le \| (\bS_\omega')^t\rvert_{W_\omega} \| \le (\rho(\omega) + \varepsilon)^t
\]
in any given operator norm for $t$ large enough. On the other hand, since $W_\omega$ is complementary to the null space $V_\omega$ of $\bP'$, there exist constants $0 < c \le C$ such that
\[
 c\| x \| \le \| \bP' x \| \le C \| x \|
\]
for all $x \in W_\omega$. We now consider an arbitrary $x \in \R^{m\times m}_{\mathrm{sym}} \times \R^{n\times n}_{\mathrm{sym}}$ and decompose into $x = x_1 + x_2$ with $x_1 \in W_\omega$ and $x_2 \in V_\omega$. Since the subspaces are complementary, there exists a uniform constant $D > 0$ such that $\| x_1 \| \le D \| x \|$. Taking everything into account, we deduce that $\| (\bP' \bS_\omega')^t x \| = \| (\bP' \bS_\omega')^t x_1 \|$ and hence
\[
 c \| (\bS_\omega')^t x_1 \| \le \| (\bP' \bS_\omega')^t x \| \le C \| (\bS_\omega')^t x_1 \| \le CD (\rho(\omega) + \varepsilon)^t \| x \|. 
\]
for $t$ large enough (independent of $x$). The upper bound in this estimate shows that $\| (\bP' \bS_\omega')^t \|^{1/t} \le (CD)^{1/t}(\rho(\omega) + \varepsilon)$. On the other hand, choosing $x_2 = 0$ and $x_1 \in W_\omega$ such that $\| (\bS_\omega')^t x_1 \| = \| (\bS_\omega')^t \| \| x \| \ge (\rho(\omega) - \varepsilon)^t \| x \|$, the lower bound implies that $\| (\bP' \bS_\omega')^t \|^{1/t} \ge c^{1/t} (\rho(\omega) - \varepsilon)$. Since $\varepsilon$ can be taken arbitrarily small,~\eqref{eq: spectral radius PS} follows.

Note that the proof only relied on the equivariance property of the fixed-point iteration and the form of the derivative $\bS_\omega'(X_*,Y_*)$. In light of Lemma~\ref{lem: equal derivatives}, we can hence obtain the same result for the iterations~\eqref{eq: overrelaxation variable transform} and~\eqref{eq: overrelaxation Hilbert metric}.
\end{proof}

\subsection{Global convergence}\label{sec: global convergence}

Since the original fixed-point iteration~\eqref{eq: full fpi} is nonlinear, it is not clear that global convergence can be preserved when relaxation is applied. For the version~\eqref{eq: overrelaxation Hilbert metric} using geodesic overrelaxation (Algorithm~\ref{algo: overrelaxation geodesic}) we are able to prove it for a range $\omega \in (0,\frac{2}{1 + \sqrt{\Lambda_1 \Lambda_2}})$. The result is obtained in a similar way as~\cite[Theorem~3]{Lehmann2022}, but first requires some additional properties of geodesics $X \#_\omega \tilde X$ on PD cones when $\omega$ is larger than one, which might be of independent interest. In particular, the following result is usually only stated for $\omega \in [0,1]$; see \cite[Corollary 1.1]{Gunawardena2003} for the Hilbert metric, \cite[Corollary~6.1.11]{Bhatia2007} for the related Riemannian metric, and \cite[Proposition~1.1.5]{Bacak2014} for general Busemann spaces. 

\begin{lemma}\label{lem:Hilbert-omega}
For $\omega \geq 0$ and $A, B, C \in \caC_m$,
\begin{align*}
 d_H(A \#_{\omega} B, C) &\leq \abs{1-\omega} d_H(A, C) + \omega d_H(B, C).
\end{align*}
\end{lemma}

To prove this lemma, we need the following property of geodesics.

\begin{lemma}\label{lem:Hilbert-delta}
For $\omega \ge 0$ and $A, B \in \caC_m$,
\begin{align*}
 d_H(A \#_\omega B, B) = \abs{1-\omega} d_H(A, B).
\end{align*}
\end{lemma}
\begin{proof}
Note that $A \#_\omega B = A^{1/2} (A^{-1/2} B A^{-1/2})^\omega A^{1/2}$. Let $M = A^{-1/2}B^{1/2}$ and $M = U \Sigma V^\top$ be the singular value decomposition of $M$, with $\Sigma = \diag(\sigma_1,\dots,\sigma_m)$. Then, 
\begin{align*}
 B^{-1/2} (A \#_\omega B) B^{-1/2} 
 &= B^{-1/2} A^{1/2} (A^{-1/2} B A^{-1/2})^\omega A^{1/2} B^{-1/2} \\
 &= M^{-1} (MM^\top)^\omega M^{-\top} \\
 &= V\Sigma^{-1} U^\top (U \Sigma^{2\omega} U^\top) U \Sigma^{-1} V^\top \\
 &= V\Sigma^{2(\omega - 1)} V^\top.
\end{align*}
This shows that the eigenvalues of $B^{-1/2} (A \#_\omega B) B^{-1/2}$ are $\sigma_i^{2(\omega -1)}$ for $i=1, \dots, m$. On the other hand, 
\begin{align*}
 B^{-1/2} A B^{-1/2} = (M^\top M)^{-1} = V \Sigma^{-2} V^\top
\end{align*}
so its eigenvalues are $\sigma_i^{-2}$. By the definition of the Hilbert metric, we hence have
\begin{align*}
 d_H(A, B) &= \log \left(\frac{\lambda_{\max}(B^{-1/2} A B^{-1/2})}{\lambda_{\min}(B^{-1/2} A B^{-1/2})} \right) = \log \left(\frac{\sigma_{\max}^2}{\sigma_{\min}^2}\right), \\
 d_H(A \#_\omega B, B) &= \log \left(\frac{\lambda_{\max}(B^{-1/2} (A \#_\omega B) B^{-1/2})}{\lambda_{\min}(B^{-1/2} (A \#_\omega B) B^{-1/2})} \right) = \log \left( \frac{\sigma_{\max}^{2\abs{\omega-1}}}{\sigma_{\min}^{2\abs{\omega-1}}} \right) = \abs{1-\omega} \log \left(\frac{\sigma_{\max}^2}{\sigma_{\min}^2} \right),
\end{align*}
which completes the proof.
\end{proof}

\begin{proof}[Proof of Lemma~\ref{lem:Hilbert-omega}]
For $\omega \in [0,1]$, the inequality is the well-known convexity of the Hilbert metric on symmetric cones; see, e.g.,~\cite[Corollary 1.1]{Gunawardena2003}. Therefore, it suffices to show the inequality for $\omega = 1 + \delta$ with $\delta > 0$. Then, using the triangle inequality and Lemma~\ref{lem:Hilbert-delta}, we have
\begin{align*}
 d_H(A \#_{\omega} B, C) 
 &\leq d_H(A \#_{\omega} B, B) + d_H(B, C) \\
 &= \delta d_H(A, B) + d_H(B, C) \\
 &\le \delta d_H(A, C) + (1+\delta) d_H(B, C),
\end{align*}
which completes the proof.
\end{proof}

Using the Lipschitz constants $\Lambda_1$ and $\Lambda_2$ (see~\eqref{eq: contraction ratio}) of the solution maps $S_1$ and $S_2$, Lemma~\ref{lem:Hilbert-omega} allows to formulate recursive error estimates for the sequence $(X_t,Y_t)$ generated by the fixed-point iteration~\eqref{eq: overrelaxation Hilbert metric}, namely
\[
 d_H(X_{t+1}, X_*) = d_H(X_t \#_\omega S_1(Y_t), X_*) \le \abs{1-\omega} d_H(X_t,X_*) + \omega \Lambda_1 d_H(Y_t,Y_*)
\]
and
\begin{align*}
 d_H(Y_{t+1}, Y_*) 
 &= d_H(Y_t \#_\omega S_2(X_{t+1}), Y_*) \\
 &\le \abs{1-\omega} d_H(Y_t,Y_*) + \omega \Lambda_2 d_H(X_{t+1},X_*) \\
 &\le \omega \Lambda_2 \abs{1-\omega} d_H(X_t,X_*) + (\abs{1-\omega} + \omega^2 \Lambda_1 \Lambda_2) d_H(Y_t,Y_*).
\end{align*}
This can be written as a recursive vector-valued inequality
\[
\begin{pmatrix}
    d_H(X_{t+1}, X_*) \\
    d_H(Y_{t+1}, X_*)
\end{pmatrix}
\le
\begin{pmatrix}
    \abs{1-\omega} & \omega \Lambda_1 \\ 
    \omega \abs{1-\omega} \Lambda_2 & \abs{1-\omega} + \omega^2 \Lambda_1 \Lambda_2
\end{pmatrix}
\cdot
\begin{pmatrix}
    d_H(X_{t}, X_*) \\
    d_H(Y_{t}, X_*)
\end{pmatrix}.
\]
A sufficient condition for global convergence (up to rescaling), that is, $d_H(X_t, X_*) \to 0$ and $d_H(Y_t, X_*) \to 0$, hence is that the spectral radius of the matrix on the right side of this inequality is less than one. A direct computation of both of its eigenvalues gives the value $\abs{1-\omega} + \frac{(\omega \Lambda)^2}{2} + \sqrt{\frac{(\omega \Lambda)^4}{4} + (\omega \Lambda)^2 \abs{1-\omega}}$ with $\Lambda = \sqrt{\Lambda_1 \Lambda_2}$ for the spectral radius, leading to the following theorem, analogous to~\cite[Theorem~3]{Lehmann2022}.

\begin{theorem}\label{thm: global convergence overrelaxation}
Under the conditions of Theorem~\ref{thm: global convergence fpi}, the iterates $(X_t,Y_t)$ of the iteration~\eqref{eq: overrelaxation Hilbert metric} (Algorithm~\ref{algo: overrelaxation geodesic}) converge to $(X_*,Y_*)$ in the Hilbert metric at a linear rate for all $\omega \in (0, \frac{2}{1 + \sqrt{\Lambda_1 \Lambda_2}})$.
\end{theorem}

The range for $\omega$ in this theorem might be very pessimistic, since $\frac{2}{1 + \sqrt{\Lambda_1 \Lambda_2}}$ can be very close to one. In particular, the optimal value $\omega_{\mathrm{opt}}$ in Theorem~\ref{thm: local convergence} is usually not covered. Whether global convergence is still ensured for larger values of $\omega$ remains an open theoretical problem.

\section{Numerical experiments}\label{sec: numerical experiments}

In this section, we present results of numerical experiments for comparing our SOR methods with the operator Sinkhorn iteration. We implemented three versions of overrelaxation, namely, PD-OR (Algorithm~\ref{algo: overrelaxation PD}), Cholesky-OR (Algorithm~\ref{algo: overrelaxation Cholesky}), and Geodesic-OR (Algorithm~\ref{algo: overrelaxation geodesic}) as well as the basic OSI (Algorithm~\ref{algo: OSI} with Cholesky factors) in Python, using Python 3.10.14, NumPy~1.26.4, and SciPy~1.10.1. Experimental results were computed on a machine with AMD EPYC 7452 CPU with 64 logical cores and \num{1007}GB memory. The fixed point versions of the standard OSI (Algorithms~\ref{algo: FFPI} and~\ref{algo: PD-FPI}) are omitted.

\subsection{Experiment setup}

We run our algorithms on two different instances of operator scaling that will be explained in the following subsections separately, using the following experiment setup in both cases.

\paragraph{Error measure.}
For each algorithm, we focus on the current scaling $\bA_t = L_t^{} \bA R_t^\top$ of the input collection  $\bA$  of the scaling problem.
In the standard OSI (Algorithm~\ref{algo: OSI}), $\bA_t$ is implicitly computed as $\bar \bA_t$ and one can obtain the corresponding scaling matrices $L_t$ and $R_t$ on the fly with \eqref{eq: formula Lt Rt}.
In order to measure the convergence error accurately, we use $\bA_t$ instead of $\bar \bA_t$ because $\bar \bA_t$ may not be in the orbit of the original input $\bA$ with respect to the left-right action as numerical errors accumulate.
In the other methods, we do not have such an issue because $\bA_t$ can be explicitly computed as $\bA_t = L_t \bA R_t$. 

For measuring convergence, we use the error measure 
\[
 \err(\bA_t) = \sqrt{\norm*{\sum_i A_{t,i}^{} A_{t,i}^\top - \frac{1}{m}I_m}_F^2 + \norm*{\sum_i A_{t,i}^\top A_{t,i}^{} - \frac{1}{n}I_n}_F^2}.
\]
It turns out that this error equals the norm of the gradient of $f$ (see~\eqref{eq: cost function}). Therefore, we call it the gradient norm.

\paragraph{Running time measure.}
We also measure the actual running time (wall time) as a more intuitive measure of time complexity. Since the running time often fluctuates over different runs even on the same instance, we run the algorithms 10 times on each instance and report the mean and standard deviation.

\paragraph{Choice of $\omega$.}
For our overrelaxation methods, we need to estimate $\omega_{\mathrm{opt}}$ based on Theorem~\ref{thm: local convergence}, which in turn requires the asymptotic convergence rate $\beta^2$ of the operator Sinkhorn iteration. However, we do not know $\beta^2$ in advance. It is hence common to use an adaptive strategy by monitoring the convergence of the standard method. In particular we adopt the proposed strategy from~\cite{Lehmann2022,Oseledets2023} (considered in similar form already in~\cite{Hageman1975}) as follows. Let $p$ be a positive integer parameter. Then we start by running our SOR methods with $\omega=1$ (which is equivalent to the operator Sinkhorn iteration) for the first $p$ iterations, and obtain an estimate ${\hat \beta}^2$ of $\beta^2$ by computing
\[
 {\hat \beta}^2 = \sqrt{\frac{\err_p}{\err_{p-2}}},
\]
where $\err_p$ denotes the gradient norm of the $p$th iterate. Then, from the $(p+1)$th iteration, we set $\omega = \hat{\omega}_{\mathrm{opt}}$ which is computed from the formula~\eqref{eq:omega-opt} with ${\hat \beta}^2$ instead of $\beta^2$. In our experiment, we take either $p = 10$ or $p = 5$. Note that in general we could use ${\hat \beta}^2 = (\frac{\err_p}{\err_{p-\ell}} )^{1/\ell}$ for positive integer $\ell \geq 2$, but $\ell = 2$ appeared to work quite reliably.

\paragraph{Initialization.} Algorithms~\ref{algo: overrelaxation PD}--\ref{algo: overrelaxation geodesic} require initial matrices. We simply used $(L_0, R_0) = (I_m, I_n)$ or $(X_0, Y_0) = (I_m, I_n)$ in the experiments.

\subsection{Experiment in frame scaling instances}

First, we conduct experiments in \emph{frame scaling} where we found overrelaxation to work extremely well. In this problem, we are given an initial frame $x_1, \dots, x_k \in \R^n$. The goal is to find a nonsingular matrix $P \in \GL_n(\R)$ and nonnegative scalars $\alpha_1, \dots, \alpha_k \geq 0$ such that the scaled frame $\tilde x_i = \alpha_i P x_i$ satisfies 
\begin{align*}
 \norm{\tilde x_i}_2 &= \frac{n}{k} \quad (i = 1, \dots, k), \\
 \sum_{i=1}^k \tilde x_i \tilde x_i^\top &= I_n.
\end{align*}
This problem has applications in signal processing~\cite{Holmes2004} and functional analysis~\cite{Barthe1998}. The essentially same problem is also known as \emph{Tyler's M-estimator} in statistics~\cite{Tyler1987,Franks2020}.

One can show that frame scaling is equivalent to operator scaling of a CP map $\Phi: \R^{n \times n}_{\mathrm{sym}} \to \R^{k \times k}_{\mathrm{sym}}$, where
\[
 \Phi(Z) = 
 \begin{pmatrix}
 x_1^\top Z x_1^{} & & \\
 & \ddots & \\
 & & x_k^\top Z x_k^{}     
 \end{pmatrix},
 \quad
 \Phi^*(Y) = \sum_{i=1}^k y_{ii}^{} x_i^{} x_i^\top,
\]
that is, the matrices $\bA = (A_1,\dots, A_k)$ are given as
\[
 A_i = e_i^{} x_i^\top.
\]
To see this, first observe that the optimal scaling matrix $L$ can be taken diagonal, since $\Phi(Z)$ is a diagonal matrix. Correspondingly, Sinkhorn type algorithms based on Cholesky decomposition will only produce diagonal $L_t$. Let $L = \diag(\ell_1, \dots, \ell_k)$ and $R \in \GL_n(\R)$. Then, $(L, R)$ is a solution to operator scaling of the above CP map $\Phi$ if and only if 
\begin{align*}
 \begin{pmatrix}
    \ell_1^2 (R x_1)^\top (R x_1) & & \\
    & \ddots & \\
    & & \ell_k^2 (R x_k)^\top (R x_k)
 \end{pmatrix} 
 &= \begin{pmatrix}
    \frac{1}{k} & & \\
    & \ddots & \\
    & & \frac{1}{k}
 \end{pmatrix}, \\
 \sum_{i=1}^k \ell_i^2 (R x_i)(R x_i)^\top &= \frac{1}{n}I_n.
\end{align*}
So $P = R$ and $\alpha_i = n\ell_i^2$ yield a solution to frame scaling, and vice versa.

\paragraph{Data generation.}
In the experiment, we set $n = 50$, $k = 55$, and generate an initial frame from the standard Gaussian distribution of $\R^n$. We run the algorithms on the same randomly generated instance for $200$ iterations where the overrelaxation is activated at iteration $p = 10$. Although this is a random instance, we report that the result is typical and observed the same behaviour on different random seeds.

\begin{figure}[t]
    \centering
    \subcaptionbox{\label{fig:frame_iter}}[.49\textwidth]{\includegraphics[width=\linewidth]{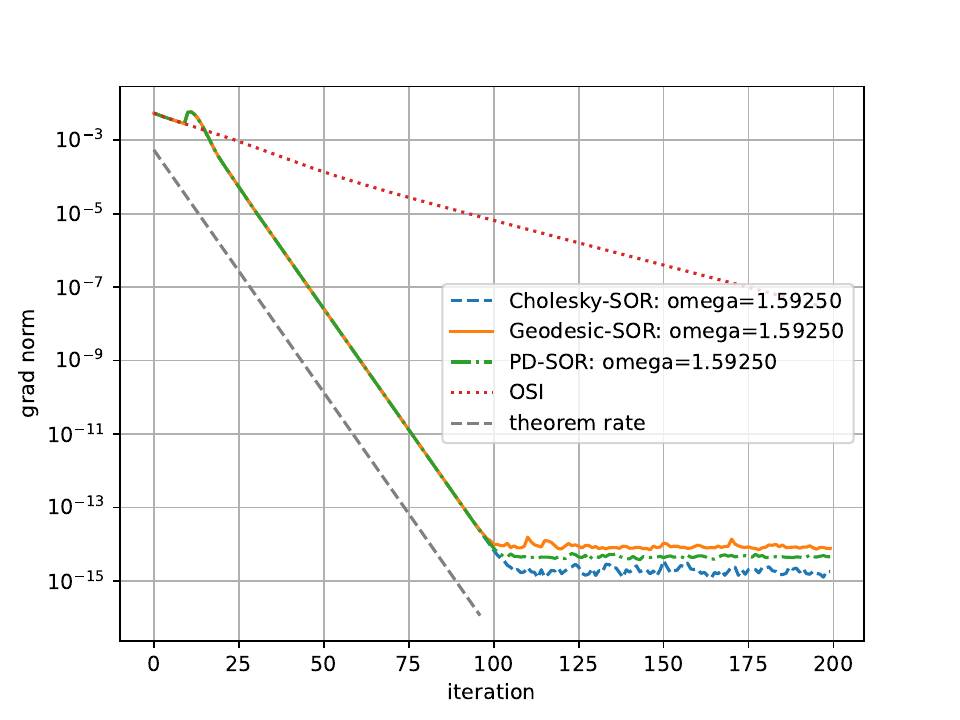}}
    \subcaptionbox{\label{fig:frame_time}}[.49\textwidth]{\includegraphics[width=\linewidth]{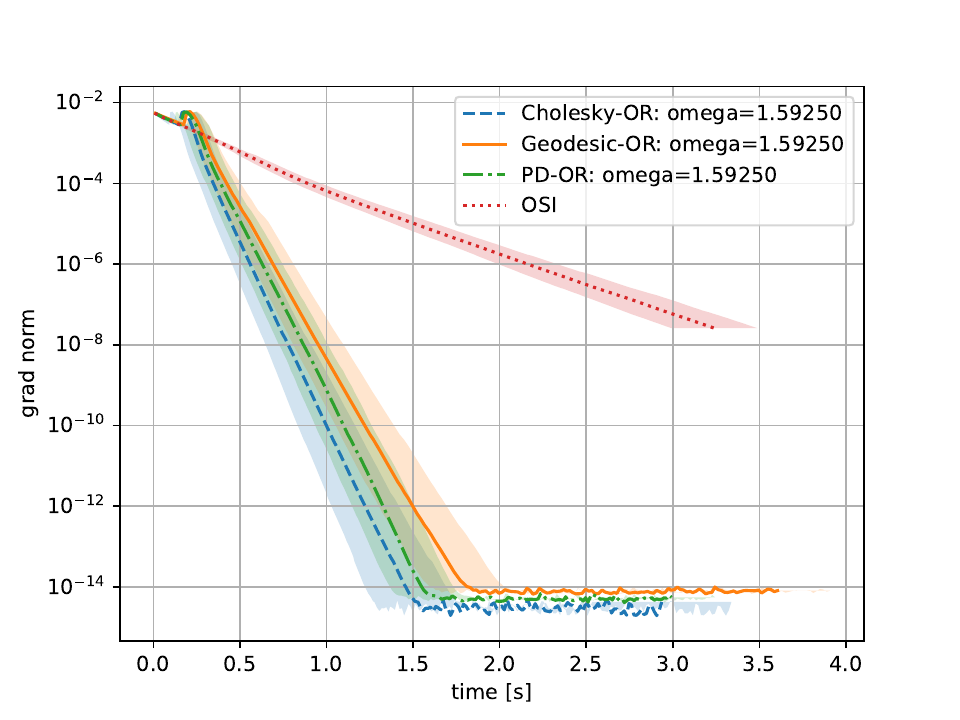}}
    \caption{Experimental results for frame scaling. The algorithms are applied to the same randomly generated instance. (\subref{fig:frame_iter})~Plot of the gradient norm against iterations. The dashed line (``theorem rate'') is the local convergence rate $\rho(\omega)$ from Theorem~\ref{thm: local convergence}. (\subref{fig:frame_time})~Plot of the gradient norm against running time. The thick line and shaded region represent the mean and the standard deviation of 10 runs on the same instance. Since both the input instance and the inital points are the same for all runs, this fluctuation is the running time fluctuation of the computing environment.\label{fig:frame}}
\end{figure}

\paragraph{Results.}
The result is shown in Figure~\ref{fig:frame}. Our SOR methods achieve an error of the order of $10^{-14}$ after about $100$ iterations, whereas the error of the operator Sinkhorn was still of the order of $10^{-8}$ after $200$ iterations; see Figure~\ref{fig:frame_iter}. We also plot the local convergence rate $\rho(\omega)$ as given in Theorem~\ref{thm: local convergence} with the actual local convergence rate $\beta^2$ of OSI. We can observe the slopes of the SOR methods match $\rho(\omega)$ quite well. We see that the convergence rates of all three SOR methods are the same, validating the local convergence result in Theorem~\ref{thm: local convergence}. We can observe a similar trend in running time (Figure~\ref{fig:frame_time}). In particular, our SOR methods achieve a desired accuracy much faster than the standard OSI. The geodesic SOR is slightly slower than the other SOR variants, which is expected because the computation of geodesics is more expensive than Cholesky decompositions.

\paragraph{Effect of $\omega$.} To see the effect of the relaxation parameter $\omega$, we also run Cholesky-OR with different fixed values of $\omega \in \{1.0, 1.2, 1.4, 1.6, 1.8\}$ (without the adaptive estimation) on the same frame scaling instance. Furthermore, we run it with the adaptively estimated parameter $\omega_{\mathrm{est}}$ and the optimal parameter $\omega_{\mathrm{opt}}$ \eqref{eq:omega-opt} with the actual local convergence rate $\beta^2$ of OSI. The result is shown in Figure~\ref{fig:varying_omega}. We can see that $\omega_{\mathrm{opt}} = 1.622$ works best as shown in Theorem~\ref{thm: local convergence}. Our adaptive estimation $\omega_{\mathrm{est}} = 1.593$ is close to $\omega_{\mathrm{opt}}$ and the convergence rate is near optimal. Although overrelaxation always accelerates OSI (i.e., $\omega = 1$), choosing $\omega$ too large (e.g., $\omega = 1.8$) leads to poor acceleration. This illustrates the importance of choosing an appropriate value of $\omega$ and shows that our adaptive estimation works quite well.

\begin{figure}[t]
    \centering
    \includegraphics[width=.5\linewidth]{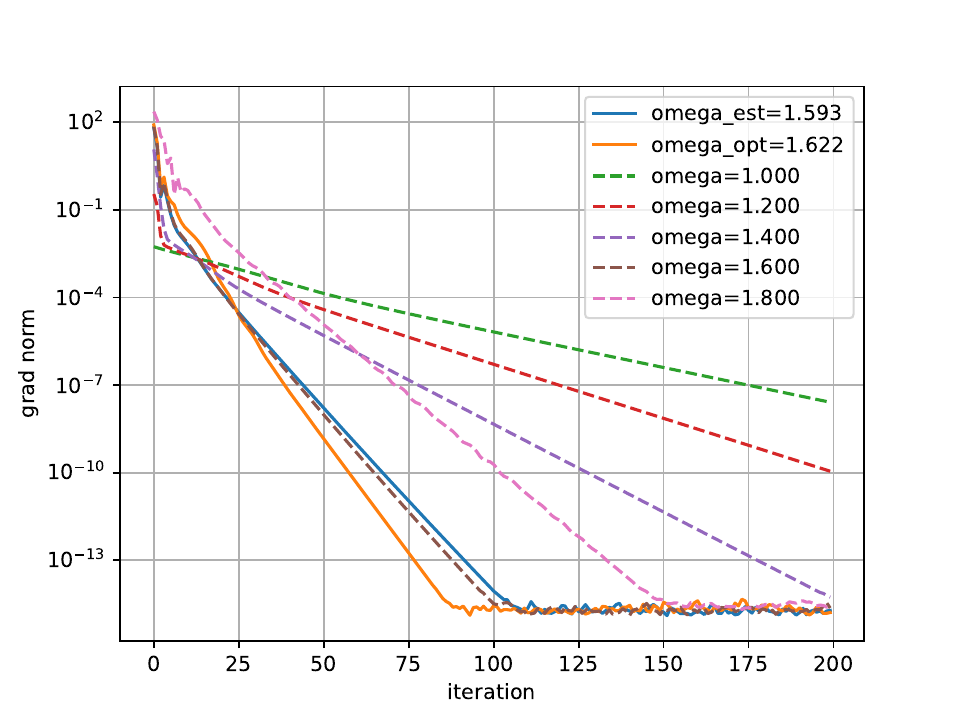}
    \caption{Effect of the relaxation parameter $\omega$. The optimal relaxation parameter \eqref{eq:omega-opt} is denoted by $\omega_{\mathrm{opt}}$ and the adaptive estimated parameter is by $\omega_{\mathrm{est}}$. The base method is Cholesky-OR.}\label{fig:varying_omega}
\end{figure}

\subsection{Ill-conditioned operator}

To observe the behaviour of the algorithms in an ill-conditioned setting and exhibit some of the limitations of overrelaxation, we consider the following artificial instances of operator scaling.

\paragraph{Data generation.}
Let $H$ be the $n \times n$ Hilbert matrix, i.e., $H_{ij} = \frac{1}{i+j-1}$ for $i,j = 1, \dots, n$. It is well-known that the Hilbert matrix is a typical example of ill-conditioned matrices.
For $i = 1, \dots, k$, we set $A_i = Q_i H \in \R^{n \times n}$, where $Q_i$ is a uniform random $n \times n$ orthogonal matrix. In the experiment, we set $n = 5$ and $k = 7$. The condition number of $A_i$ is about~\num{476600}. We run the algorithms on the same random instance for $50$ iterations where overrelaxation is activated at iteration $p = 5$.

\paragraph{Results.}
The result is shown in Figure~\ref{fig:Hilbert}. Again, the overrelaxation accelerates the convergence rate after activated, although not as much as in the previous experiment. However, our SOR methods stagnate at an accuracy around $10^{-6}$, whereas OSI achieves the accuracy of order of $10^{-11}$. The running time of OSI and the SOR methods are almost the same, but Geodesic-OR is slightly slower. Therefore, in this experiment overrelaxation is less efficient.

The degraded accuracy of the SOR methods may be explained as follows. The main step in the algorithms is the computation of Cholesky decompositions and inverses of intermediate matrices which are sums of Gramians. In OSI these are Gramians of $\bar A_{i,t} = L_t A_i R_t^\top$. Upon convergence to solutions of the operator scaling problem, the sum of these Gramians converge to well conditioned (normalized) identity matrices. This gradually compensates (to a certain limit) for the bad condition of the initial $A_i$. In contrast, in the SOR methods, which are based on either Algorithm~\ref{algo: FFPI} or~\ref{algo: PD-FPI}, sums of Gramians of only ``half-scaled'' matrices such as $A_i^{} R_t^\top$ or $L_{t+1} A_i$ appear as intermediate matrices. Hence the ill-conditioning of the $A_i$ is never fully compensated. This may also explain, why OSI approximately achieves the double numerical accuracy compared to the SOR methods ($10^{-11}$ vs. $10^{-6}$). Note that Algorithm~\ref{algo: FFPI} and~\ref{algo: PD-FPI} without overrelaxation would suffer from the same limitation of the numerical accuracy in this example (results not shown). This suggests that it would be beneficial to combine OSI (Algorithm~\ref{algo: OSI}) directly with overrelaxation. This is left for future work.

\begin{figure}[t]
    \centering
    \subcaptionbox{\label{fig:Hilbert_iter}}[.49\textwidth]{\includegraphics[width=\linewidth]{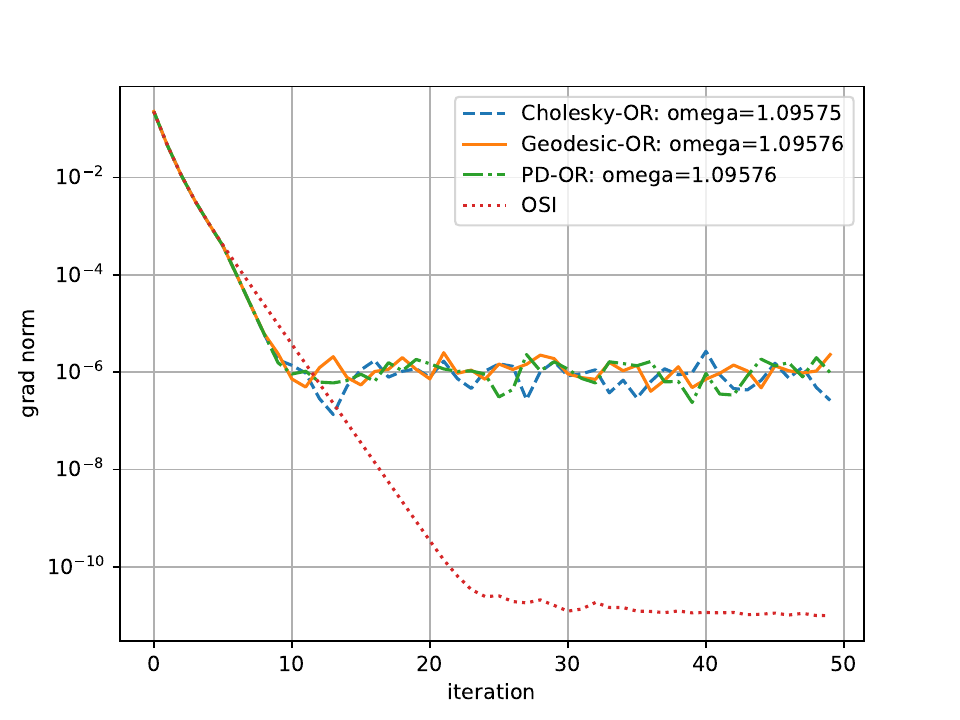}}
    \subcaptionbox{\label{fig:Hilbert_time}}[.49\textwidth]{\includegraphics[width=\linewidth]{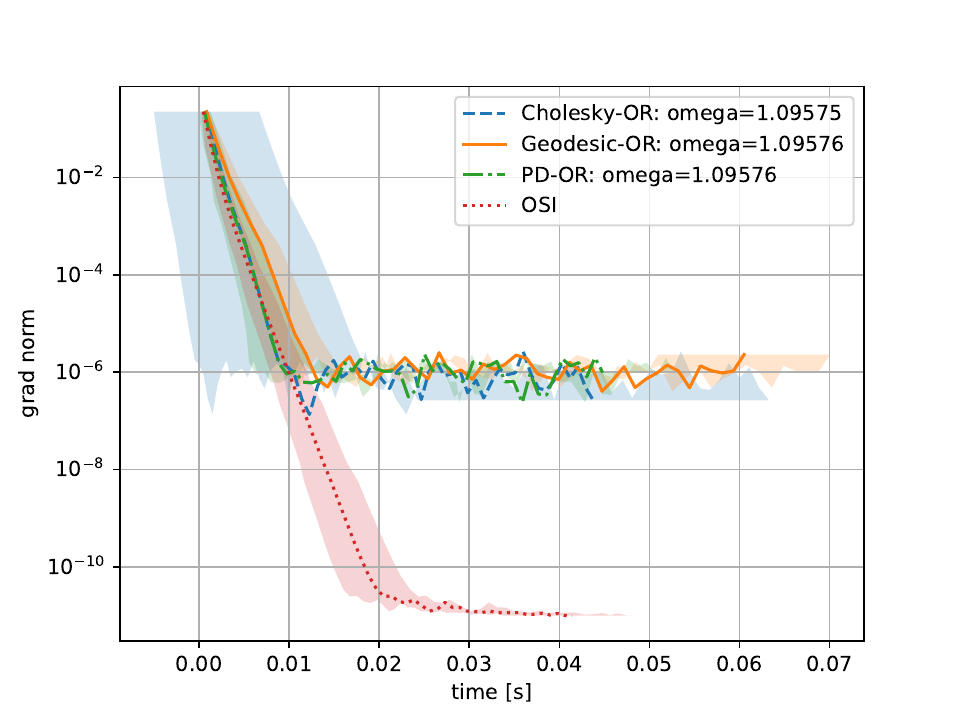}}
    \caption{Experimental results for ill-conditioned operators. (\subref{fig:Hilbert_iter}) Plot of the gradient norm against iterations. (\subref{fig:Hilbert_time}) Plot of the gradient norm against running time. The thick line and shaded region represent the mean and the standard deviation of 10 runs on the same instance. Since both the input instance and the inital points are the same for all runs, this fluctuation is the running time fluctuation of the computing environment.\label{fig:Hilbert}}
\end{figure}

\section{Conclusion}\label{sec: conclusion}
We investigated different variants of overrelaxation for the operator Sinkhorn iteration based on its interpretation as an alternating fixed-point scheme on cones of symmetric positive definite matrices. The numerically most efficient version operates directly on the Cholesky factors (Algorithm~\ref{algo: overrelaxation Cholesky}), whereas the theoretically most appealing version uses overrelaxation along geodesics in the affine invariant metric on positive definite matrices (Algorithm~\ref{algo: overrelaxation geodesic}). For this geodesic version we were able to establish global convergence in a limited (and presumably too pessimistic) range of relaxation parameters for operator scaling problems with a positivity improving CP map. However, near a fixed-point all proposed versions were shown to be equivalent to first-order, and the local convergence rate of all of them could hence be deduced by analyzing the linearized SOR iteration, taking their orbital equivariance into account. Numerical experiments illustrate that overrelaxation can significantly accelerate the difficult operator scaling task at practically no additional cost (at least for Algorithm~\ref{algo: overrelaxation Cholesky}) using an easy heuristic for choosing an almost optimal relaxation parameter. However, on ill-conditioned instances our current versions of overrelaxation may suffer from a limited numerical accuracy. Addressing this issue will be one direction for future work. Another one is to establish the global convergence for a larger range of relaxation parameters by exploiting the geodesic convexity of the underlying problem more explicitly.

\appendix
\section{Appendix}

\subsection{Proof of Proposition~\ref{prop: equivalence of algorithms}}\label{sec: proof proposition}

Obviously, for $t=0$ we have
\[
 \bar C_0^{} \bar C_0^\top = \sum_{i=1}^k A_i^{} A_i^\top = C_0^{} C_0^\top,
\]
and hence $C_0 = \bar C_0^{} P_1^\top$ for an orthogonal $P_1$ (see, e.g.,~\cite[Lemma~2.1]{Burer2005}). As a result,
\[
 L_1 = P_1 \bar L_1.
\]
This then also entails
\[
 \bar D_0^{} \bar D_0^\top = \sum_{i=1}^k A_i^\top \bar L_1^\top \bar L_1^{} A_i = \sum_{i=1}^k A_i^\top L_1^\top L_1^{} A_i = D_0^{} D_0^\top,
\]
hence $D_0 = \bar D_0^{} Q_1^\top$ with orthogonal $Q_1$, or
\[
 R_1 = Q_1 \bar R_1.
\]
We have thus shown
\[
 \bA_1 = L_1 \bA R_1^\top = P_1 \bar L_1 \bA \bar R_1^\top Q_1^\top = P_1 \bar{\bA}_1 Q_1^\top.
\]

We proceed with the induction step $t \to t+1$. By induction hypothesis $\bA_t = P_t^{} \bar{\bA}_t^{} Q_t^\top$, or equivalently, $\bar{\bA}_t = P_t^\top \bA_t^{} Q_t$, we have
\begin{align*}
 \bar C_t^{} \bar C_t^{\top} 
 = \sum_{i=1}^k \bar A_{t,i}^{} \bar A_{t,i}^\top 
 &= P_t^{\top} \left( \sum_{i=1}^k A_{t,i}^{} A_{t,i}^\top \right) P_t^{} \\
 &= P_t^{\top} L_t^{} \left( \sum_{i=1}^k A_{i}^{} R_t^\top R_t^{} A_{i}^\top \right) L_t^\top P_t^{} 
 = P_t^{\top} L_t^{} C_t^{} (P_t^{\top} L_t^{} C_t^{})^\top.
\end{align*}
This gives
\[
 P_t^{\top} L_t^{} C_t^{} = \bar C_t^{} P_{t+1}^\top
\]
for an orthogonal matrix $P_{t+1}$, and hence
\[
 C_t^{} = L_{t}^{-1} P_t \bar C_t^{} P_{t+1}^\top.
\]
Therefore,
\[
 L_{t+1} 
 = \frac{1}{\sqrt{m}} C_t^{-1} 
 = \frac{1}{\sqrt{m}} (L_{t}^{-1} P_t \bar C_t^{} P_{t+1}^\top)^{-1} 
 = \frac{1}{\sqrt{m}} P_{t+1}^{} \bar C_t^{-1} P_t^{-1} L_t 
 = P_{t+1}^{} \bar L_{t+1} P_t^{\top} L_t,
\]
or equivalently
\[
 \bar L_{t+1} = P_{t+1}^\top L_{t+1} L_t^{-1} P_{t}^{}.
\]
Using again the induction hypothesis this next yields $\bar L_{t+1} \bar{\bA}_t = P_{t+1}^\top L_{t+1}^{} \bA R_t^\top Q_t^\top$ and thus
\[
 \bar D_t^{} \bar D_t^\top = \sum_{i=1}^k \bar A_{t,i}^\top \bar L_{t+1}^\top \bar L_{t+1}^{} \bar A_{t,i} = Q_t R_t \left(\sum_{i=1}^k A_i^\top L_{t+1}^\top L_{t+1}^{} A_i^{} \right) R_t^\top Q_t^\top = (Q_t^{} R_t^{} D_t^{}) (Q_t^{} R_t^{} D_t^{})^\top.
\]
This shows $Q_t^{} R_t^{} D_t^{} = \bar D_t^{} Q_{t+1}^\top$, or
\[
 R_{t+1} = \frac{1}{\sqrt{n}} D_t^{-1} = \frac{1}{\sqrt{n}} Q_{t+1}^{} \bar D_t^{-1} Q_t R_t
 = Q_{t+1}^{} \bar R_{t+1}^{} Q_t R_t.
\]
Putting things together, we finally obtain
\begin{align*}
 \bA_{t+1} 
 = L_{t+1} \bA R_{t+1}^\top 
 &= P_{t+1} \bar L_{t+1} P_t^{\top} L_t \bA R_t^\top Q_t^\top \bar R_{t+1}^\top Q_{t+1}^\top \\ 
 &= P_{t+1} \bar L_{t+1} \bar{\bA}_t \bar R_{t+1}^\top Q_{t+1}^\top = P_{t+1} \bar{\bA}_{t+1} Q_{t+1}^\top,
\end{align*}
concluding the induction step.\hfill  

\subsection{Proof of Lemma~\ref{lem: derivative of geodesic}}\label{sec: proof lemma}

The Taylor series $x^\omega = \sum_{k=0}^\infty \binom{\omega}{k}(x-1)^k$ converges for $\abs{x-1} < 1$, where $\binom{\omega}{k} = \frac{\omega (\omega - 1) \dots (\omega - k + 1)}{k!}$ is the generalized binomial coefficient. Since for sufficiently small perturbations $H$ and $\tilde H$ any eigenvalue $\lambda$ of $(X + H)^{-1/2}(X+\tilde H)(X+H)^{-1/2}$ satisfies $\abs{ \lambda - 1} < 1$, we can apply the Taylor series to this matrix (see, e.g.,~\cite[Theorem~4.7]{Higham2008}), which gives
\[
 [(X + H)^{-1/2}(X+\tilde H)(X+H)^{-1/2}]^\omega = \sum_{k=0}^\infty \binom{\omega}{k}[ I - (X + H)^{-1/2}(X+\tilde H)(X+H)^{-1/2} ]^k.
\]
Since for $k \ge 2$ we have
\begin{align*}
 [I - (X + H)^{-1/2}(X+\tilde H)(X+H)^{-1/2}]^k &= [I - (X^{-1/2} + O(\|H\|))(X+\tilde H)(X^{-1/2}+O(\|H\|)]^k \\ &= [I - X^{-1/2} X X^{-1/2} + O(\| H \|) + O(\| \tilde H \|) ]^k \\ &= O(\| (H,\tilde H) \|^2),
\end{align*}
we obtain the first-order expansion 
\[
 [(X + H)^{-1/2}(X+\tilde H)(X+H)^{-1/2}]^\omega = I + \omega [I - (X + H)^{-1/2}(X+\tilde H)(X+H)^{-1/2}] + O(\| (H,\tilde H) \|^2),
\]
where $O(\|(H,\tilde H)\|)^2$ denotes terms of at least quadratic order in $H$ and $\tilde H$. Let
\[
 (X+H)^{-1/2} = X^{-1/2} + D[H] + O(\| H \|^2)
\]
with $H \mapsto D[H]$ denoting the Fr\'echet derivative of $A \mapsto A^{-1/2}$ at $X$. Inserting this in the above expression gives us the first-order expansion
\begin{equation}\label{eq: expansion omega}
\begin{aligned}
 &[(X + H)^{-1/2}(X+\tilde H)(X+H)^{-1/2}]^\omega \\
 {}={} &I + \omega (D[H] X^{1/2} + X^{1/2} D[H] + X^{-1/2}\tilde H X^{-1/2}) + O(\| (H,\tilde H) \|^2).
\end{aligned}
\end{equation}
We next use that
\begin{equation}\label{eq: expansion square root}
 (X + H)^{1/2} = X^{1/2} - X^{1/2} D[H] X^{1/2} + O(\| H \|^2).
\end{equation}
This relation between the derivative of $X^{1/2}$ and $X^{-1/2}$ can be derived by differentiating the identity $X^{1/2} X^{-1/2} = I$. Inserting~\eqref{eq: expansion square root} and~\eqref{eq: expansion omega} into the definition of $(X+H)\#_{\omega} (X+\tilde H)$ and rearranging gives
\[
 (X+H)\#_{\omega} (X+\tilde H) = X + \omega \tilde H - (1-\omega)(X D[H] X^{1/2} + X^{1/2}D[H]X) + O(\| (H,\tilde H) \|^2).
\]
From differentiating the equation $X \cdot (X^{-1/2})^2 \cdot X = X$ using product and chain rule, one can deduce the identity
\[
 X D[H] X^{1/2} + X^{1/2}D[H]X = - H, 
\]
which completes the proof. 

\subsection{Geodesic convexity}\label{sec:g-convex-f}

Geodesic convexity (\emph{g-convexity} in the following) generalizes the concept of convexity from Euclidean spaces to Riemannian manifolds. We refer to~\cite[Chapter~11]{Boumal2023} and~\cite{Udriste1994} for details on g-convexity. Here, we focus on the specific manifold $\caC_m$ of $m \times m$ symmetric positive definite matrices, endowed with the \emph{affine invariant metric}
\[ 
 g(H, \tilde H)_{X} = \tr(H X^{-1} \tilde H X^{-1}),
\]
where $X \in \caC_m$ and $H, \tilde H \in \R^{m \times m}_{\mathrm{sym}}$. 
In this metric, $\caC_m$ becomes uniquely geodesic and the geodesic between two points $X, \tilde X \in \caC_m$ is precisely given by $X \#_\omega \tilde X$, $\omega \in [0,1]$, as defined in~\eqref{eq:sharp}~\cite[Theorem~6.1.6]{Bhatia2007}. A function $f: \caC_m \to \R$ is said to be g-convex (with respect to the affine-invariant metric) if
\[
 f(X \#_\omega \tilde X) \leq (1-\omega)f(X) + \omega f(\tilde X)
\]
for any $X, \tilde X \in \caC_m$ and $\omega \in [0,1]$. Functions for which always equality holds in this definition are said to be g-affine. The following lemma is useful to check the g-convexity.

\begin{lemma}[{\cite[Lemma~3.6]{Duembgen2016}}]\label{lem:g-convex}
 A function $f: \caC_m \to \R$ is g-convex if and only if for any $B \in \GL_m(\R)$ the function $f_B(x) = f(B\diag(e^x)B^\top)$ is convex in $x \in \R^m$, where 
 \[
 \diag(e^x) = \begin{pmatrix}
 e^{x_1} & & \\
 & \ddots & \\
 & & e^{x_m}
 \end{pmatrix}. 
\]
\end{lemma}

We say that $f : \caC_m \times \caC_n \to \R$ is jointly g-convex if for $X, \tilde X \in \caC_m$ and $Y, \tilde Y \in \caC_n$, we have $f(X \#_\omega \tilde X, Y \#_\omega \tilde Y) \leq (1-\omega)f(X, Y) + \omega f(\tilde X, \tilde Y)$ for $\omega \in [0,1]$. In other words, $f$ is g-convex in the product manifold $\caC_m \times \caC_n$. With a similar idea as in the proof of Lemma~\ref{lem:g-convex}, one can show the following.

\begin{lemma}\label{lem:jointly g-convex}
A function $f: \caC_m \times \caC_n \to \R$ is jointly g-convex if and only if for any $(B, C) \in \GL_m(\R) \times \GL_n(\R)$, the function $f_{B,C}(x, y) = f(B\diag(e^x)B^\top, C\diag(e^y)C^\top)$ is jointly convex in $(x, y) \in \R^m \times \R^n$. 
\end{lemma}

Here is an example of jointly g-convex functions.

\begin{lemma}\label{lem:g-convex-XAYA}
Let $A \in \R^{m \times n}$. Then $f(X, Y) = \tr(XAYA^\top)$ is jointly g-convex in $\caC_m \times \caC_n$.
\end{lemma}

\begin{proof}
By Lemma~\ref{lem:jointly g-convex}, it suffices to show that the function
\[ 
 f_{B,C}(x,y) = \tr(B\diag(e^x)B^\top A C\diag(e^y) C^\top A^\top) 
\]
is jointly convex in $(x,y) \in \R^m \times \R^n$ for any $(B, C) \in \GL_m(\R) \times \GL_n(\R)$. Let $D = B^\top A C \in \R^{m \times n}$ so that $f_{B,C}(x,y) = \tr(\diag(e^x) D \diag(e^y) D^\top)$. Then, $(D \diag(e^y) D^\top)_{i,j} = \sum_{k=1}^n d_{ik}e^{y_k}d_{jk}$ for $i, j = 1, \dots, m$. So $f_{B,C}(x,y) = \sum_{i=1}^m \sum_{k=1}^n d_{ik}^2 e^{x_i + y_k}$, which is evidently jointly convex in $(x,y)$.
\end{proof}

Here is the main result of this section.

\begin{lemma}\label{lem:g-convex-f}
Let $f: \caC_m \times \caC_n \to \R$ be the function \eqref{eq: cost function}, namely, 
\[ 
 f(X,Y) = \tr(X\Phi(Y)) - \frac{1}{m}\log\det(X) - \frac{1}{n}\log\det(Y),
\]
where $\Phi: \R^{m \times m}_{\mathrm{sym}} \to \R^{n \times n}_{\mathrm{sym}}$ is a CP map. Then $f$ is jointly g-convex in $(X,Y)$.
\end{lemma}

\begin{proof}
Since $\log\det(X)$ and $\log\det(Y)$ are g-affine, it suffices to show that $\tr(X \Phi(Y))$ is g-convex. By writing $\Phi(Y) = \sum_{i=1}^k A_i^{} Y A_i^\top$, this follows from Lemma~\ref{lem:g-convex-XAYA}.
\end{proof}

\subsection*{Acknowledgements}
T.S.~was supported by JSPS KAKENHI Grant Numbers 19K20212 and 24K21315, and JST, PRESTO Grant Number JPMJPR24K5, Japan. The work of A.U.~was supported by the Deutsche Forschungs\-gemeinschaft (DFG, German Research Foundation) – Projektnummer 506561557. This work benefited from several valuable discussions with Max von Renesse which is gratefully acknowledged.

\small
\bibliographystyle{alpha}
\bibliography{main}

\end{document}